\documentclass[Afour,sagev,times]{sagej}
\pdfoutput=1

\usepackage{moreverb,url}

\usepackage[colorlinks,bookmarksopen,bookmarksnumbered,citecolor=red,urlcolor=red]{hyperref}

\usepackage{graphicx}
\usepackage[per-mode=fraction]{siunitx}
\usepackage{multirow}
\usepackage{eurosym}
\usepackage{pifont}
\usepackage[utf8]{inputenc}
\usepackage{mathtools}
\usepackage{amsfonts}
\usepackage{paralist}
\usepackage{textcomp,gensymb}
\usepackage{subcaption}
\usepackage{float}
\usepackage[english]{babel}
\usepackage{multicol}
\usepackage[utf8]{inputenc}
\usepackage{wrapfig}
\usepackage{pdfpages}
\usepackage[dvipsnames]{xcolor}
\usepackage{fancyhdr}
\usepackage{algorithm}
\usepackage{nicefrac}
\usepackage{algpseudocode}
\usepackage{svg}

\newtheorem{theorem}{Theorem}[section]
\newtheorem{definition}{Definition}[section]
\newtheorem{remark}{Remark}[section]
\newtheorem{lemma}{Lemma}[section]

\algtext*{EndWhile}
\algtext*{EndIf}
\algtext*{EndFor}
\algtext*{Procedure}
\algtext*{EndProcedure}

\newcommand\BibTeX{{\rmfamily B\kern-.05em \textsc{i\kern-.025em b}\kern-.08em
T\kern-.1667em\lower.7ex\hbox{E}\kern-.125emX}}

\def\volumeyear{2023}

\begin{document}

\runninghead{Wotte et al.}

\title{Optimal Potential Shaping on SE(3) via Neural ODEs on Lie Groups}

\author{Yannik P. Wotte \affilnum{1},       
        Federico Califano  \affilnum{1},    
        Stefano Stramigioli \affilnum{1}}

\affiliation{\affilnum{1}RaM, University of Twente, NL} 

\corrauth{Yannik P. Wotte, RaM, University of Twente, Drienerlolaan 5, 7522NB Enschede, NL.}

\email{y.p.wotte@utwente.nl}

\begin{abstract}
This work presents a novel approach for the optimization of dynamic systems on finite-dimensional Lie groups. 
We rephrase dynamic systems as so-called neural ordinary differential equations (neural ODEs), and formulate the optimization problem on Lie groups. 
A gradient descent optimization algorithm is presented to tackle the optimization numerically. Our algorithm is scalable, and applicable to any finite dimensional Lie group, including matrix Lie groups. By representing the system at the Lie algebra level, we reduce the computational cost of the gradient computation.
In an extensive example, optimal potential energy shaping for control of a rigid body is treated. 
The optimal control problem is phrased as an optimization of a neural ODE on the Lie group $SE(3)$, and the controller is iteratively optimized. 
The final controller is validated on a state-regulation task.
\end{abstract}

\keywords{Nonlinear Control, Deep Learning, Differential Geometry}

\maketitle


\section{Introduction}\label{sec:int}

\begin{figure*}
    \centering
    \includegraphics[width=0.9\textwidth]{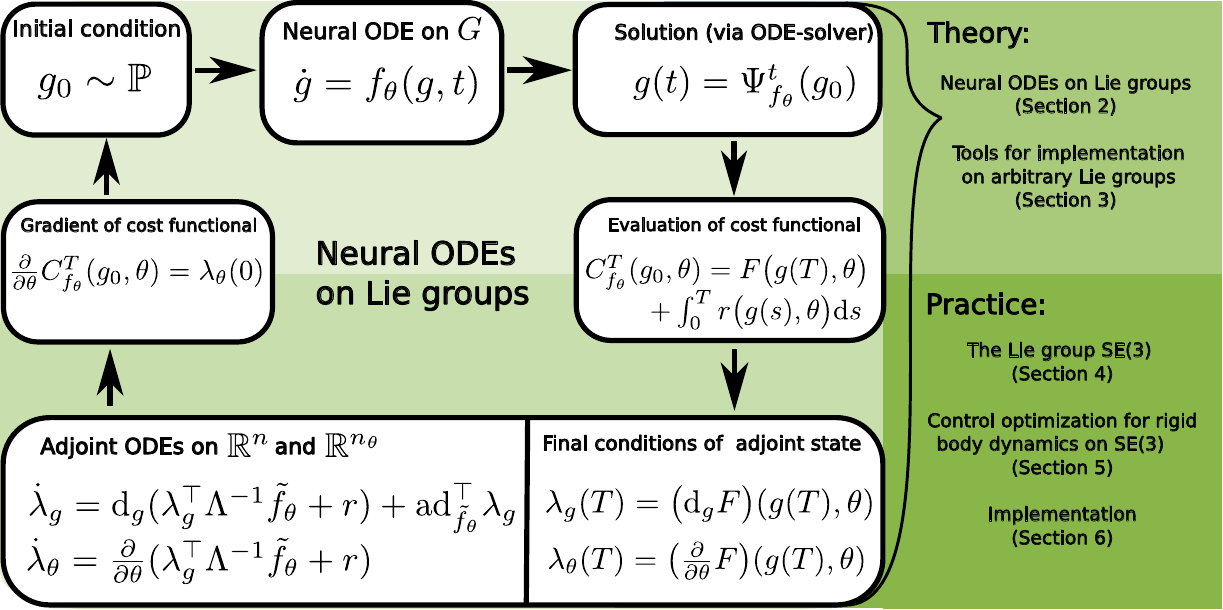}
    \caption{Overview of the main contribution and structure of the article. Given a parameterized dynamical system on a Lie group, the generalized adjoint method on Lie groups lets us compute the parameter gradient of a cost-functional over system trajectories by solving a set of differential equations. This parameter gradient can then be used to iteratively update parameters by gradient descent. In practice, we sample multiple initial conditions and approximate the parameter gradient of the expected cost $C^T_{f_\theta}(\theta) := \mathbb{E}_{g_0\sim\mathbb{P}} C_{f_\theta}^T(g_0,\theta)$.}
    \label{fig:overview}
\end{figure*}

Many physical systems are naturally described by the action of Lie groups on their configuration manifolds. This can range from finite-dimensional systems such as rigid bodies, where poses are acted on by the special Euclidean group $SE(3)$ \citep{Murray1994}, towards infinite-dimensional systems such as flexible bodies or fluid dynamical systems, where the diffeomorphism group acts on the configuration of the continuum \citep{Schmid2010}. 

Geometric control systems on Lie groups \citep{Brockett1973,Jurdjevic1996} exploit the Lie group structure of the underlying physical systems to provide numerical advantages \citep{Marsden1999}.
For example, \textit{PD} controllers for rigid bodies were defined on $SO(3)$ and $SE(3)$ by \cite{bullo_murray_1995}, and more recently geometric controllers were applied in the context of UAV's \citep{Lee2010,Goodarzi2013,Rashad2019}. 
Examples for efficient optimal control formulations on Lie groups include linear \citep{Ayala2021} and nonlinear systems \citep{Spindler1998}, as well as efficient numerical optimization methods \citep{Kobilarov2011,Saccon2013,Teng2022}.

In an orthogonal development over the recent years, there has been a surge of machine learning applications in control \citep{Dev_ML_Control_Review_2021} and robotics \citep{Taylor2021,Ibarz2021,Soori_AI_Robotics_Review_2023}. This surge is driven by the need for controllers that work in high-dimensional robotic systems and approximate complex decision policies that require the use of data. The implementation of such controllers through classical control theoretic approaches is prohibitive, and it led to a paradigm shift towards data-driven control \citep{Taylor2021}. Examples of machine learning within high-dimensional systems extend to soft robotics \citep{Kim_ML_Soft_Robotics_Review_2021} and control of fluid systems \citep{paris_RL_flow_control_2021}. The literature also aims to address common concerns of safety \citep{Hewing2020,Brunke2022} both during the training process and in the deployment of systems with machine learning in the loop.

The so-called Erlangen program of machine learning by \cite{Bronstein2021} stresses the importance of geometric machine learning methods: symmetries of data sets can restrict the complexity of functions that are to be learned on them, and thus increase the numerical efficiency of learning frameworks. This rationale also led to extensions of machine learning approaches to Lie groups \citep{Fanzhang2019,Lu2020}, with recent applications by \cite{Huang2017,Chen2021,Forestano2023}. 
 
Indeed, the fundamental symmetry groups in robotics are naturally represented by Lie groups \citep{Marsden1999}. As such, Lie group-based learning methods are of interest to the robotics community. In an excellent example of a control application \cite{Duong2021_nODE_SE3} extended neural ODEs to $SE(3)$ and applied it to the adaptive control of a UAV in  \cite{Duong2021_Adaptive_control_via_nODE_SE3}. In their recent work \cite{duong2024porthamiltonian} also highlight the practical use of neural ODEs on Lie groups. 

However, a general approach for geometric machine learning in the context of dynamic systems on Lie groups is missing. We believe that such an approach would be of high interest, especially for control applications. In this paper, we address this issue by formalizing neural ODEs on Lie groups.

Our contributions are 
\begin{enumerate}
    \item the formulation of neural ODEs on any finite-dimensional Lie group, with a particular focus on matrix Lie groups;
    \item computational simplifications with respect to manifold neural ODEs through use of a compact equation to compute gradients on Lie groups, and reduced dimension with respect to non-intrinsic approaches on Lie groups
    \item a \texttt{pytorch} \& \texttt{torchdyn} compatible algorithm for the optimization of a general potential energy shaping and damping injection control on $SE(3)$, for which stability is implemented as a design requirement; available at \href{https://github.com/YPWotte/Lie_nODEs}{\texttt{github.com/YPWotte/Lie\_nODEs}}.
    \item the formulation of a minimal exponential Atlas on the Lie group $SE(3)$.
\end{enumerate}
 
\sloppy
The article is divided into two parts (see also Figure \ref{fig:overview}): first the formulation of neural ODEs on finite-dimensional matrix Lie groups, and second an extensive example of optimal potential energy shaping on $SE(3)$. 

Section \ref{sec:main_result} presents the main technical contribution of the article: the generalized adjoint method on Lie groups, which is at the heart of a gradient descent algorithm for dynamics optimization via neural ODEs on Lie groups. 

A number of technical tools are required to apply this algorithm on a given matrix Lie group, which are introduced in Section \ref{sec:technical_tools}. The exponential Atlas allows to implement a numerical procedure for exact integration on Lie groups, while a compact formula for the gradient of a function on a Lie group reduces complexity of the gradient computation.

Section \ref{sec:kin} presents the Lie groups $SO(3)$ and $SE(3)$, and gives concrete examples of the technical tools presented in the previous section. One aspect of this is the formulation of a minimal exponential Atlas on $SE(3)$, which is used to formulate an integration procedure on $SE(3)$. This treatment prepares the stage for control optimization of a rigid body on $SE(3)$. 

Section \ref{sec:opt} introduces the example of optimizing potential energy and damping injection controllers for rigid bodies on $SE(3)$. The class of controllers is defined and it is shown that it guarantees stability by design. Afterwards, the optimization of a cost-functional over the defined class of controllers is derived from the general procedure in Section \ref{sec:main_result}.

Finally, Section \ref{sec:train} provides two examples of optimizing controllers for a rigid body on $SE(3)$. The first example concerns pose control, without gravity, and results are compared to a quadratic controller of the type presented by \cite{Rashad2019}. 
In the second example, the controller's performance is investigated in the presence of gravity. 

The article ends with a discussion in Section \ref{sec:discussion} and a conclusion in Section \ref{sec:conclusion}. 

\subsection{Neural ODEs and relation to existing works}
Neural ODEs were first introduced by \cite{chen2019neural}, who derived them as the continuous limit of recurrent neural nets, taking inputs on $\mathbb{R}^n$. Their cost functionals only admitted intermediate and final cost terms, for which they showed that the so-called adjoint method allows a memory-efficient computation of the gradient.

\cite{Massaroli2021} introduced a more general framework of neural ODEs, showing the power of state-augmentation and connections to optimal control, while also showing that the cost functional can include integral cost terms. To this end, they presented the \textit{generalized} adjoint method.

There are two highly relevant examples in the recent literature that extend neural ODEs to manifolds. The so-called extrinsic picture is presented by \cite{Falorsi2020}, who show that neural ODEs on a manifold $\mathcal{M}$ can be optimized as classical neural ODEs on an embedding $\mathbb{R}^n$. Given an extension of the manifold neural ODE to $\mathbb{R}^n$, they show that the adjoint method on $\mathbb{R}^n$ can be applied for optimization of the manifold neural ODE.

An intrinsic picture is presented by \cite{Lou2020}, who show that neural ODEs on a manifold $\mathcal{M}$ can be expressed in local charts on the manifold, where the adjoint method holds locally. They use exponential charts on Riemannian manifolds, and achieve a dimensionality-reduction and geometric exactness with respect to Falorsi et al. Both \cite{Falorsi2020} and \cite{Lou2020} carefully extend neural ODEs 
to manifolds, and consider neural ODE on Lie groups to be a sub-class of the presented manifold neural ODEs. With respect to their work, we show how to include integral costs in a \textit{generalized} adjoint method on manifolds and Lie groups, and show the advantages of considering neural ODEs on Lie groups as a specialized class of algorithms. 

An example of neural ODEs to control of robotic systems described on $\mathbb{R}^n$ is described in \cite{Massaroli2020}, where an IDA-PBC controller is optimized.

\cite{Duong2021_Adaptive_control_via_nODE_SE3,Duong2021_nODE_SE3} apply neural ODEs to control optimization for a rigid body on $SE(3)$. The work focuses on the formulation of an IDA-PBC controller, uses it for dynamics learning and trajectory tracking, and uses neural ODEs as a tool for this optimization.
While the integration procedure used is not geometrically exact and the Lie group constraints are violated, the approach is highly successful. However, Duong et al. do not connect their contribution to geometric machine learning literature such as neural ODEs on manifolds. In recent work \cite{duong2024porthamiltonian}, steps are made to extend the extrinsic approach to general matrix Lie groups, however without making full use of the geometric structure given by Lie groups.

With respect to \cite{duong2024porthamiltonian}, we present neural ODEs on arbitrary finite dimensional Lie groups. By extending the intrinsic formulation to Lie groups, our example on $SE(3)$ has a reduced number of dimensions (24 instead of 36), and the use of local charts allows geometrically exact integration.


\subsection{Notation}\label{ssec:notation} 
While the main results are accessible with a background of linear algebra and vector calculus, the derivations heavily rely on differential geometry and Lie group theory, see e.g, \cite{Isham1999} and \cite{Hall2015} for a complete introduction, or \cite{sola2021micro} for a brief introduction with examples in robotics.

Calligraphic letters $\mathcal{M},\,\mathcal{N},\,\mathcal{U},\,\mathcal{P}$ denote smooth manifolds.
Respectively, $T_x\mathcal{M}$ and $T_x^*\mathcal{M}$ denote the tangent and cotangent space at $x \in \mathcal{M}$,  $T\mathcal{M} $ and $T^*\mathcal{M}$ denote the tangent bundle and cotangent bundle of $\mathcal{M}$, and $\Gamma(T\mathcal{M})$ and $\Gamma(T^*\mathcal{M})$ are the sets of sections that collect vector fields and co-vector fields over $\mathcal{M}$. 
Curves $x:\mathbb{R}\rightarrow \mathcal{M}$ are evaluated as $x(t)$, and their tangent vectors are denoted as $\dot{x} \in T_{x(t)} \mathcal{M}$. 

Upper case letters $G,H$ denote Lie groups, while lower case letters $g,h$ denote their elements. A lower case $e$ denotes the group identity $e\in G$, an upper case $I$ denotes the identity matrix. The Lie algebra is $\mathfrak{g}$, and its dual is $\mathfrak{g}^*$. Letters $\Tilde{A},\Tilde{B}$ denote vectors in the Lie algebra, while letters $A,B$ denote vectors in $\mathbb{R}^n$. 

%
Furthermore $C^k(\mathcal{M},\mathcal{N})$ denotes the set of continuous, $k$-times differentiable functions between $\mathcal{M}$ and $\mathcal{N}$.  
For $\phi \in C^k(\mathcal{M},\mathcal{N})$, let $\phi_*: \Gamma(T\mathcal{M}) \rightarrow \Gamma(T\mathcal{N})$ and $\phi^*: \Gamma(T^*\mathcal{N}) \rightarrow \Gamma(T^*\mathcal{M})$ denote the push-forward and pullback, respectively. 

For $V\in C^1(\mathcal{M},\mathbb{R})$, let $\text{d}V \in \Gamma(T^*\mathcal{M})$ denote the gradient co-vector field. When $\mathcal{M} = \mathbb{R}^k$, the gradient at $x\in \mathbb{R}^k$ is denoted by $\frac{\partial V}{\partial x} \in \mathbb{R}^k$. 

When coordinate expressions are concerned, the Einstein summation convention is used, i.e., the product of variables with lower and upper indices implies a sum $a_i b^i := \sum_i a_i b^i$.

Let $(X,D,\mathbb{P})$ denote a probability space with $X$ a topological space, $D$ the Borel $\sigma$-algebra and $\mathbb{P}:D\rightarrow[0,1]$ 
a probability measure. 
Given a vector space $L$ and a random variable $C:X\rightarrow L$, denote by $\mathbb{E}_{x\sim\mathbb{P}}(C) := \int_X C(x) \text{d}\mathbb{P}(x)$ 
the expectation of $C$ w.r.t. $\mathbb{P}$. 

\section{Main Result}\label{sec:main_result}

After a brief introduction to Lie groups in Section \ref{ssec:main_result/lie_groups}, the optimization problem is introduced on abstract Lie groups in Section \ref{ssec:main_result/optimization_problem}. A gradient descent optimization algorithm is presented in Section \ref{ssec:main_result/main_result}. Our main technical result, the generalized adjoint method on Lie groups, lies at the core of the gradient computation. For the sake of exposition, we present it in the context of matrix Lie groups, and relegate the derivations and the formulation on abstract Lie groups to Appendix \ref{app:A}.

\subsection{Lie groups}\label{ssec:main_result/lie_groups}
A finite-dimensional Lie group $G$ is an $n$-dimensional manifold together with a group structure, such that the group operation is a smooth map on $G$ \citep{Isham1999}. $G$ is a real matrix Lie group if it is a subgroup of the general linear group $GL(m,\mathbb{R})$\endnote{A more general definition of a matrix lie group $G$ allows for complex matrix Lie groups $G\subset GL(m,\mathbb{C})$ or quaternionic matrix Lie groups $G\subset GL(m,\mathbb{H})$. The results in our article immediately extend to such scenarios: a choice of basis for the Lie algebras $gl(m,\mathbb{C}),\, gl(m,\mathbb{H})$ leads to $\Lambda:\mathbb{R}^n \rightarrow \mathfrak{g}$ in Equation \eqref{eq:abstract_tilde}, such that relevant quantities like the adjoint map in \eqref{eq:algebra_adjoint}, the adjoint state and its dynamics in Theorem \ref{thm:adjoint_method_matrix_lie_group} may again be expressed as real valued vectors and matrices. }  
\begin{equation}
    GL(m,\mathbb{R}) := \{g \in \mathbb{R}^{m\times m} \; | \; \text{det}(g) \neq 0 \}\,,
\end{equation}
where the group operation for a matrix Lie group is given by matrix multiplication \citep{Hall2015}. For $g,h \in G$ the left translation by $h$ is defined as
\begin{equation}
    L_h:G\rightarrow G\,; \; L_h(g) := hg\,.
\end{equation}

We denote the Lie algebra of $G$ as $\mathfrak{g} := T_e G$, and its dual as $\mathfrak{g}^* := T^*_e G$. 

Define a basis $E := \{\Tilde{E}_1,\ldots,\Tilde{E}_n\}$ with $\Tilde{E}_i \in \mathfrak{g}$, and define the (invertible, linear) map $\Lambda:\mathbb{R}^n\rightarrow \mathfrak{g}$ as\endnote{When directly working with matrix Lie groups (e.g, \cite{sola2021micro}) $\Lambda$ and $\Lambda^{-1}$ are often denoted as the so-called ``hat'' $\wedge: \mathbb{R}^n\rightarrow\mathbb{R}^{m\times m}$ and ``vee'' $\vee:\mathbb{R}^{m\times m} \rightarrow \mathbb{R}^n$ operators, respectively.}

\begin{equation} \label{eq:abstract_tilde}
    \Lambda:\mathbb{R}^n \rightarrow \mathfrak{g} \,
; \; (x^1,\ldots,x^n)\mapsto \sum_i x^i \Tilde{E}_i \,.
\end{equation}

The dual of $\Lambda$ is the map $\Lambda^*:\mathfrak{g}^*\rightarrow \mathbb{R}^n$. 
%
Define the dual basis $\{\bar{E}^1,\ldots,\bar{E}^n\}$ with $\bar{E}^i \in \mathfrak{g}^*$ by $\bar{E}^i(\tilde{E}_j) = \delta_j^i$ with $\delta_j^i$ the Kronecker delta. Then ${\Lambda^*}^{-1}$ is explicitly given by
\begin{equation}
    {\Lambda^*}^{-1}:\mathbb{R}^n \rightarrow \mathfrak{g}^* \,
; \; (x_1,\ldots,x_n)\mapsto \sum_i x_i \bar{E}^i \,.
\end{equation}

For a matrix Lie group the Lie algebra $\mathfrak{g}$ is a subspace of the Lie algebra $gl(m,\mathbb{R})$ of $GL(m,\mathbb{R})$. Here $gl(m,\mathbb{R})$ is defined as 
\begin{equation}
        gl(m,\mathbb{R}) := \mathbb{R}^{m\times m}\,.
\end{equation}

For $\Tilde{A},\Tilde{B} \in \mathfrak{g}$ the adjoint map $\text{ad}_{\Tilde{A}}(\Tilde{B})$ is a bilinear map defined in terms of the (left) Lie bracket
\begin{equation}\label{eq:algebra_adjoint} 
    \text{ad}:\mathfrak{g}\times\mathfrak{g}\rightarrow\mathfrak{g}\,;\; \text{ad}_{\tilde{A}}(\tilde{B}) = \Tilde{A}\Tilde{B} - \Tilde{B}\Tilde{A}\,.
\end{equation} 

Using the operator $\Lambda$, a matrix representation of $\text{ad}$ is obtained as $\Lambda^{-1}\big(\text{ad}_{\Lambda(A)}\Lambda(\cdot)\big)\in\mathbb{R}^{n\times n}$, called the adjoint representation. By an abuse of notation, we denote the adjoint representation as $\text{ad}_{A}$, without a tilde in the subscript.

On matrix Lie groups and for functions $V \in C^1(G,\mathbb{R})$ the gradient $\text{d}_g V \in \mathbb{R}^n$ (see Section \ref{ssec:technical_tools:gradients} for details) is found as:
\begin{equation}
    \text{d}_g V = \frac{\partial}{\partial q} V\bigg(g\big(I+\Lambda(q)\big)\bigg)_{|q = 0}\,.
\end{equation}


\subsection{Optimization problem} \label{ssec:main_result/optimization_problem}

We consider a variant of the optimal control problem on a Lie group \citep{Jurdjevic1996} with a finite horizon $T$. Given parameters $\theta \in \mathbb{R}^{n_\theta}$, denote the parameterized dynamics on a Lie group as $f_\theta(g,t) := f(g,t,\theta)$. Then, given the dynamic system 
\begin{equation}\label{eq:dyn_Lie}
    \dot{g} = f_\theta(g,t)\,, \; g(0) = g_0 \,,
\end{equation}
denote the solution operator (also called the flow) as 
\begin{equation}\label{eq:flow}
    \Psi^t_{f_\theta}:G \rightarrow G\, ; \; g(0) \mapsto g(t) \,,
\end{equation}
and define the real valued cost function
\begin{equation}\label{eq:single_trajectory_cost}
    C^T_{f_\theta}(g_0,\theta) = F(\Psi^T_{f_\theta}(g_0),\theta) + \int_0^T r(\Psi^s_{f_\theta}(g_0),\theta,s) \text{d}s \,,
\end{equation}
where we call $F$ 
the final cost term and $r$ 
the running cost term.

Indicating a probability space $(G,D,\mathbb{P})$, we are interested in solving the minimization problem
\begin{align} \label{eq:abstract_optimization_problem}
    \min_\theta \; & \mathbb{E}_{g_0 \sim \mathbb{P}} C^T_{f_\theta}(g_0,\theta)\,.
\end{align}

\begin{remark}
    The chief reason for our interest in this optimization problem is that it includes, as a subclass, the optimization of state-feedbacks $u_\theta:G\times [0,T] \rightarrow\mathcal{U}$ by considering dynamics of the form $f(g,t,u_\theta(g,t))$, where $u_\theta$ denotes the control input of the system. 
\end{remark}

\begin{remark} 
    The dynamics $f_\theta(g,t)$ can also be parameterized with neural nets, in which case $f_\theta(g,t)$ is referred to as a neural ODE on a Lie group. Indeed, for the Lie group $(\mathbb{R}^n,+)$, the formulation agrees with the definition of a neural ODE given in \cite{Massaroli2021} who define them as dynamics $\dot{x} = f_\theta(x,t)$ with $x \in \mathbb{R}^n$.
\end{remark}

\subsection{Optimization algorithm}  \label{ssec:main_result/main_result}

We use a stochastic gradient descent optimization algorithm \citep{Robbins1951} to approximate a solution to the optimization problem \eqref{eq:abstract_optimization_problem} on a matrix Lie group. 

Denote the total cost in \eqref{eq:abstract_optimization_problem} as
\begin{equation}
    J(\theta) := \mathbb{E}_{g_0 \sim \mathbb{P}} C^T_{f_\theta}(g_0,\theta) \,.
\end{equation}

Additionally, denote by $\theta_k \in \mathbb{R}^{n_\theta}$ the parameters at the $k$-th iteration, and by $\eta_k$ a positive scalar learning rate at the $k$-th iteration. Then a \textit{standard} gradient descent algorithm computes the parameters $\theta_{k+1}$ by an application of the update rule
\begin{equation} \label{eq:gradient_descent}
    \theta_{k+1} = \theta_k - \eta_k \frac{\partial}{\partial\theta} J(\theta) \,. 
\end{equation}
In \textit{stochastic} gradient descent $N$ initial conditions $g_i$ are sampled from the probability distribution corresponding to the probability measure $\mathbb{P}$. The expectation in \eqref{eq:gradient_descent} is approximated by averaging the gradients of costs $C_i = C^T_{f_\theta}(g_i,\theta)$ of the individual trajectories starting at $g_i$ as 
\begin{equation}
    \frac{\partial}{\partial\theta} J(\theta) = \mathbb{E}_{g_0 \sim \mathbb{P}} \frac{\partial}{\partial\theta} C^T_{f_\theta}(g_0) \approx \frac{1}{N} \sum_{i=0}^N \frac{\partial}{\partial\theta} C_i \,.
\end{equation}
For convex cost-functions $J(\theta)$ and a sufficiently small $\eta_k$, the parameter $\theta_k$ approaches the optimal parameters $\theta^\star$ as $k$ increases \citep{Robbins1951}. For non-convex cost-functions stochastic gradient descent does not have a guarantee of global optimality, but it is still widely used as a light and scalable algorithm \citep{ruder2017overview} that results in robust local optima \citep{xie2021diffusion}. 

In order to compute the gradient $\frac{\partial}{\partial\theta} C_i$ of the cost for a single trajectory \eqref{eq:single_trajectory_cost}, we derived the generalized adjoint method on matrix Lie groups. It is the main technical result of this paper, and it is stated in the following:  

\begin{theorem}[Generalized Adjoint Method on Matrix Lie Groups]\label{thm:adjoint_method_matrix_lie_group}
    Given are the dynamics \eqref{eq:dyn_Lie} 
    and the cost 
    \eqref{eq:single_trajectory_cost}. 
    Denote by $\tilde{f}_\theta(g,t) := g^{-1} f_\theta(g,t) \in \mathfrak{g}$. Then the parameter gradient $\frac{\partial}{\partial\theta} C^T_{f_\theta}(g_0)$ of the cost is given by the integral equation
    \begin{equation} \label{eq:parameter_gradient_matrix_lie_group}
        \frac{\partial}{\partial\theta} C^T_{f_\theta}(g_0) = \frac{\partial F}{\partial \theta} + \int_0^T\frac{\partial}{\partial \theta} \big(\lambda_{g}^\top \Lambda^{-1}(\tilde{f}_\theta)\big) + \frac{\partial r}{\partial \theta} \text{d}t\,,
    \end{equation}
    where the state $g(t) \in G$ and adjoint state $\lambda_g(t) \in \mathbb{R}^n$ are the solutions of the system of equations
    \begin{align}
        \dot{g} &= f_\theta\,, \; g(0) = g_0 \,, \label{eq:Mat_Lie_Adjoint_Sensitivity} \\
        \dot{\lambda}_{g} &= - \text{d}_g\big(\lambda_g^\top\Lambda^{-1}(\tilde{f}_\theta)+r\big) 
        + \text{ad}_{\tilde{f}_\theta}^\top \lambda_g \,, \; \lambda_{g}(T) = \text{d}_g F\label{eq:Mat_Lie_Adjoint_Sensitivity_Co}\,.
    \end{align}
\end{theorem}

\begin{proof}
    For a derivation of Theorem \ref{thm:adjoint_method_matrix_lie_group} via neural ODEs on manifolds, refer to Appendix \ref{app:A}.
\end{proof}

The generalized adjoint method \citep{Massaroli2021} on $\mathbb{R}^n$ is recovered as a special case for the Lie group $(\mathbb{R}^n,+)$, for which the adjoint term $\text{ad}_T = 0$ such that equation \eqref{eq:Mat_Lie_Adjoint_Sensitivity_Co} agrees with the adjoint equation on $\mathbb{R}^n$. 

Just as the generalized adjoint method on $\mathbb{R}^n$, the generalized adjoint method on Lie groups has a constant memory efficiency with respect to the network depth $T$. This makes it an advantageous choice for the gradient computation compared to e.g., back-propagation through an ODE solver.

Various technical tools are required to apply Theorem \ref{thm:adjoint_method_matrix_lie_group} in practice. This includes the exponential Atlas for exact integration of $\dot{g}$, a tractable expression of the gradient operator $\text{d}_g:C^1(G,\mathbb{R}) \rightarrow \mathbb{R}^n$, and the composition of matrix Lie groups to create new matrix Lie groups from old ones. These tools are the subject of Section \ref{sec:technical_tools}.

\begin{remark}
    Equations \eqref{eq:Mat_Lie_Adjoint_Sensitivity} and \eqref{eq:Mat_Lie_Adjoint_Sensitivity_Co} are solved by integrating \eqref{eq:Mat_Lie_Adjoint_Sensitivity} forward in time, computing $\text{d}_g F$ at $g = g(T)$, and integrating \eqref{eq:Mat_Lie_Adjoint_Sensitivity_Co} backwards in time, reusing $g(t)$ from the forward integration. Equation \eqref{eq:parameter_gradient_matrix_lie_group} is solved by integrating its differential alongside Equation \eqref{eq:Mat_Lie_Adjoint_Sensitivity_Co}. See especially Figure \ref{fig:overview}. The memory efficiency of neural ODEs stems from the fact that trajectories $g(t)$ and $\lambda_g(t)$ do not need to be stored, apart from a few way-points of $g(t)$, and that the dependency of $g(t)$ on parameters is largely ignored in the forward pass - this avoids overheads that arise e.g., through automatic differentiation over an ODE solver.
\end{remark}

\begin{remark}
    The choice of group action for $(\mathbb{R}^n,\bullet)$ plays an important role in recovering the adjoint equation on $\mathbb{R}^n$. This choice is a nontrivial degree of freedom of the optimization (see also Remark \ref{remark:group_action_se*3}).
\end{remark}

\section{Technical Tools}\label{sec:technical_tools}

A number of technical tools are presented in the context of matrix Lie groups. Given mild adaptations of the definitions these tools also apply to abstract finite-dimensional Lie groups (see Appendix \ref{app:nodes_on_manifolds}). 


\subsection{Atlas and minimal Atlas on Lie groups}\label{ssec:technical_tools:atlas}

\begin{figure}
 \centering
 \includegraphics[height=23em]{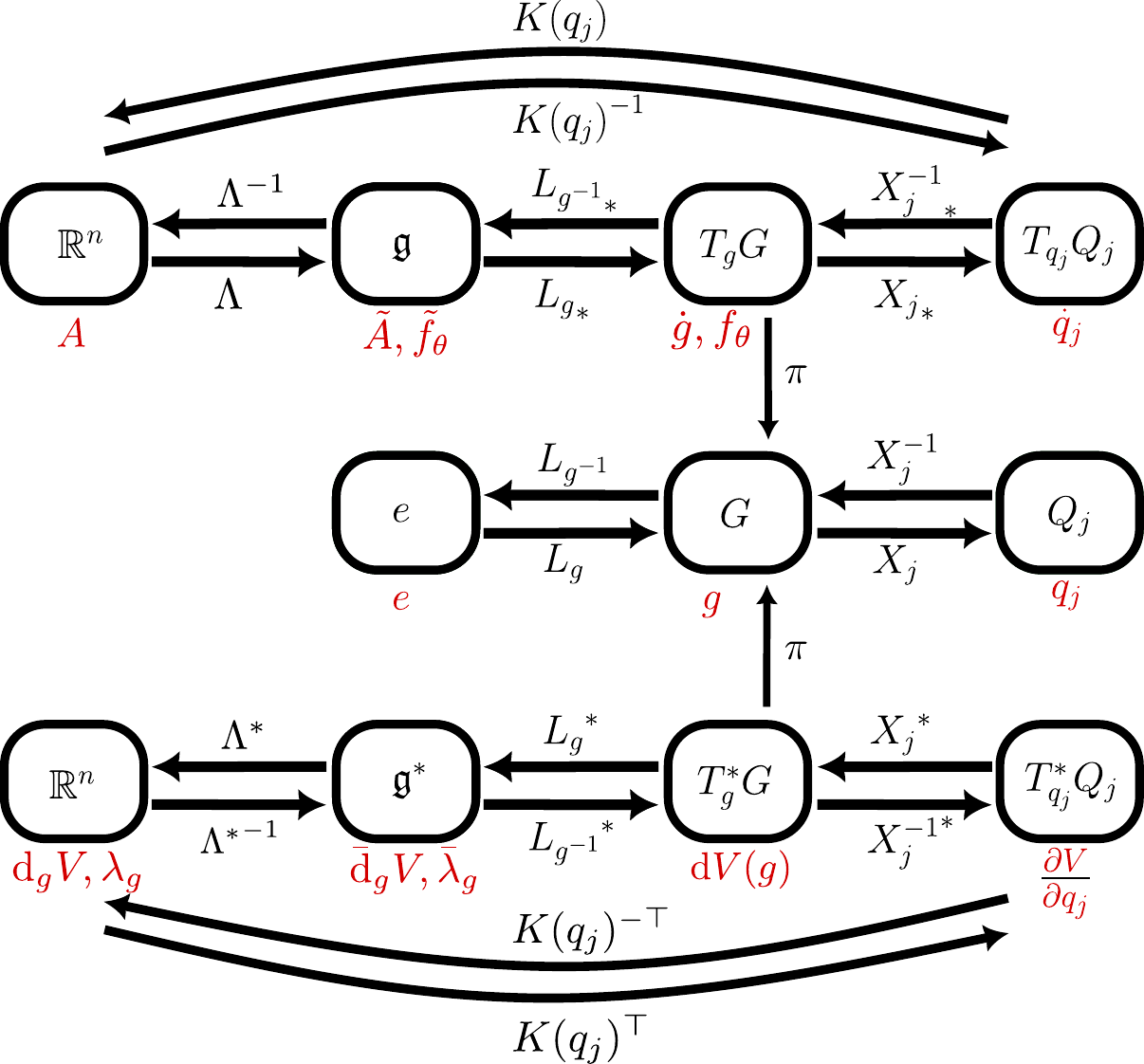}
 \caption{Commutative diagram of a generic Lie group $G$. 
 Boxes represent sets, while arrows represent functions between sets. Relevant variables in a given set are indicated in red.}\label{fig:CommDiag_G}
\end{figure}

In this section the exponential map and logarithmic maps will be used to construct an atlas of \textit{exponential} charts for finite-dimensional Lie groups, and the concept of a minimal exponential atlas will be defined. Here an atlas is defined as follows:

\begin{definition}[Atlas and Charts]
An \textbf{atlas} $\mathcal{A}$ for an $n$-dimensional smooth manifold $\mathcal{M}$ is a collection of \textbf{charts} $(U,X)$, where $U \subseteq \mathcal{M}$ is an open set, $X:U\rightarrow\mathbb{R}^n$ is a diffeomorphism called a \textbf{chart map}, and the chart domains satisfy $\bigcup_{(U,X)\in\mathcal{A}} U = \mathcal{M}$. 
\end{definition}

For finite-dimensional Lie groups the 
exponential map $\exp:\mathfrak{g}\rightarrow G$ is a local diffeomorphism \citep[Chapter 4.2.3]{Isham1999}. 
Its inverse $\log:U\rightarrow\mathfrak{g}$ is defined by $\exp\circ\log = \text{id}_{U}$, for a neighborhood $U$ of the identity $e\in G$, and $\text{id}_U$ is the identity map on $U$. 

For a matrix Lie group, 
the exponential map is given by the infinite sum \citep[Chapter 3.7]{Hall2015}:
\begin{equation}\label{eq:general_matrix_exp}
    \exp(\tilde{A}) := \sum_{n=0}^\infty \frac{1}{n!}\tilde{A}^n\,.
\end{equation}

Conversely, the $\log$ map for matrix Lie groups is given by the matrix logarithm, when it is well-defined \citep[Chapter 2.3]{Hall2015}: 
\begin{equation}\label{eq:general_matrix_log}
    \log(g) = \sum_{n=1}^{\infty} (-1)^{n+1} \frac{(g-I)^n}{n}\,.
\end{equation}

On a case-by-case basis the infinite sums in \eqref{eq:general_matrix_exp} and \eqref{eq:general_matrix_log} can further be reduced to a finite sum by use of the Cayley-Hamilton Theorem \citep{Visser2006}, which often allows one to find a closed-form expression of the $\exp$ and $\log$ maps.

The logarithmic map \eqref{eq:general_matrix_log} and $\Lambda$ in equation \eqref{eq:abstract_tilde} can then be used to construct a local exponential chart $(U,X)$ for $G$, where 
\begin{align}
    X:\,&U\rightarrow\mathbb{R}^n \,;
    \;
     g \mapsto \Lambda^{-1}\log(g)  \,, 
\end{align}
assigns so-called coordinates $q \in \mathbb{R}^n$ to group elements $g\in U \subseteq G$, with the zero coordinates assigned to the group identity $e$.

To create a chart ``centered" on any $h\in G$ (i.e. the zero coordinates are assigned to $h$), both the region $U$ and the chart map $X$ can be left-translated\endnote{The left-translation of a map $f:G\rightarrow\mathcal{M}$ is defined as $(L_g f) := f\circ L_{g^{-1}}$.} by $L_{h}$ to define the chart $(U_h,X_h)$ with 
\begin{align}
    U_h &= \, \{h g \; | \; g \in U  \} \,,\label{eq:one_chart_G}\\
    X_h&:\,U_h\rightarrow\mathbb{R}^n \,;\label{eq:one_chart_G_X} 
    \; g \mapsto \Lambda^{-1}\log(h^{-1}g) \,, 
    \\ 
    X_h^{-1}&:\,\mathbb{R}^n \rightarrow G \,; \label{eq:one_chart_G_X_inv} 
    \; q \mapsto h\exp\big(\Lambda({q})\big) \,. 
\end{align}

The collection $\mathcal{A}$ of charts $(U_h,X_h)$ is then called an exponential atlas. This atlas covers the Lie group $G$, and is fully determined by the choice of basis $E \subset \mathfrak{g}$ and chart region $U \subset G$. 

In order to use a finite number of charts, we are interested in constructing a minimal exponential atlas. A minimal atlas is defined as follows: 
\begin{definition}[Minimal Atlas]
    An atlas $\mathcal{A}$ is minimal if it covers the manifold, i.e. $\mathcal{M} = \bigcup_{(U,X)\in\mathcal{A}} U$, and if it does so with the minimum number of charts. 
\end{definition}

\begin{remark}  
    Given a manifold $\mathcal{M}$, the size of a minimal atlas is determined by a topological invariant, the integer 
    $\text{Cat}(\mathcal{M})$ (the Lusternik-Schnirrelmann category, see \cite{Grafarend2011,Oprea2014}). 
    Given a Lie group $G$, $\text{Cat}(G)$ provides a \textit{lower bound} on the size of a minimal \textit{exponential} atlas. 
\end{remark}

\begin{remark}
Given a minimal exponential atlas (which can be seen to be always countable) we use integers $j$ to number the relevant chart-centers as $g_j$, the corresponding charts as $(U_j,X_j)$, and denote the chart-coordinates in the $j$-th chart as $q_j \in \mathbb{R}^n$.
\end{remark}

\subsection{Vectors on Lie groups}\label{ssec:technical_tools:vectors} 

Given a curve
$g:\mathbb{R} \rightarrow G$ and its coordinate representation $q_j(t) = X_j(g(t))$ through \eqref{eq:one_chart_G_X}, one can relate the tangent-vector $\dot{q}_j \in T_{q_j} Q_j$ of $q_j\in Q_j \subseteq \mathbb{R}^n$ to an element $\Tilde{A}\in\mathfrak{g}$, and further to an element $A = \Lambda^{-1}(\Tilde{A}) \in \mathbb{R}^n$ (see Fig. \ref{fig:CommDiag_G}) by 

\begin{equation}\label{eq:G_vec_to_algebra_rep}
    A = K(q_j) \dot{q}_j\,.
\end{equation}
The map $K:\mathbb{R}^n \rightarrow \mathbb{R}^{n\times n}$ 
is called the derivative of the exponential map, and it is given by the power series \citep{rossmann2006lie}
\begin{equation}\label{eq:dexp_rep}
    K(q_j) = \sum_{k=0}^\infty \frac{(-1)^k}{(k+1)!}\text{ad}_{q_j}^k\,.
\end{equation}

Recall that $\text{ad}_{q_j}$ is an $n$-by-$n$ matrix, and $\text{ad}_{q_j}^k$ is the $k$-th power of such matrices. As with the matrix exponential \eqref{eq:general_matrix_exp}, the infinite sum can then be reduced to a finite sum by use of the Cayley-Hamilton Theorem \citep{Visser2006}.

\begin{remark} The expression \eqref{eq:G_vec_to_algebra_rep} is invariant under choice of exponential chart, i.e., for any charts $(U_j,X_j)$ and $(U_k,X_k)$ one has that $K(q_j)\dot{q}_j = K(q_k)\dot{q}_k$ (see Appendix \ref{appA:chart_invariance_K}).
\end{remark}
\medskip

\subsection{Lie group integrators}\label{ssec:technical_tools:lie_group_integrators}
We adapt the Runge-Kutta-Munthe-Kaas (RKMK) method \citep{MuntheKaas1999} for exact integration of the dynamics \eqref{eq:dyn_Lie}. The RKMK method uses the Runge-Kutta scheme on Lie groups - we instead allow for arbitrary numerical integration schemes. 
For an overview of Lie group integrators, see e.g. \cite{iserles_2000,Celledoni_2003,Celledoni_2014}.

Using equation \eqref{eq:G_vec_to_algebra_rep}, the dynamics \eqref{eq:dyn_Lie} can be represented in a local exponential chart as \citep{Celledoni_2014}
\begin{equation}\label{eq:chart_dyn_G}
    \dot{q}_j = K(q_j)^{-1} \Lambda^{-1} \tilde{f}_\theta\big( X_j^{-1}(q_j), t\big) \,, 
\end{equation}
where we denote $\tilde{f}_\theta(g,t) := g^{-1} f_\theta(g,t) \in \mathfrak{g}$.
An arbitrary numerical integration scheme can then be used to integrate the dynamics, as long as the state $q_j$ remains in the chart region $Q_j = X_j(U_j) \subseteq \mathbb{R}^n$. 

\begin{remark}
    Note that the chart-dynamics \eqref{eq:chart_dyn_G} are not necessarily well-defined for all $q_j \in \mathbb{R}^n$, since $K(q_j)$ can have singularities. Yet, the chart-dynamics are well-defined for $q_j$ in the chart region $Q_j$, where $K(q_j)$ is guaranteed to be of full rank. 
\end{remark}

To make sure the state remains within the chart region we switch charts when needed by application of
\begin{equation}
    q_i = X_i\big(X_j^{-1}(q_j)\big)\,.
\end{equation}
Here the conditions for chart switches are a degree of freedom. It is possible to always choose the chart $(X_h,U_h)$ with $h = X_j^{-1}\big(q_j(t)\big)$, i.e., to switch charts after every step of the numerical integrator as in \cite{MuntheKaas1999,Celledoni_2014}. 

Given a minimal exponential atlas, we choose to reduce the amount of chart switches. 
To this end, we introduce indicator functions $\sigma_j:G\rightarrow \mathbb{R}$ s.t. $\sigma_j(g) > 0$ if $g\in U_j$, and switch charts when $\sigma_j(g)$ is smaller than a threshold value (see Appendix \ref{ssec:app_ChartTransitions}).

\begin{remark}
    The use of a minimal exponential atlas allows to store way-points of a trajectory $g(t) \in G \subset GL(m,\mathbb{R})$ in terms of the $n+1$ numbers $(q_j(t),j)$, rather than storing the $m\times m$ entries of $g$. In principle, this more memory-efficient storing does not require an integration procedure based on a minimal atlas, since it can be implemented as a post-processing step given the trajectory $g(t)$. While not done in this work, it is worth consideration to use standard numerical integrators such as RKMK, which come with well-known error-bounds \citep{MuntheKaas1999}. 
\end{remark}

\subsection{Gradients on Lie groups}\label{ssec:technical_tools:gradients}

The gradient of a function $V\in C^1(G,\mathbb{R})$ is the co-vector field $\text{d}V \in \Gamma(TG)$. For any given $g\in G$ the gradient $\text{d}V(g) \in T_g^*G$ is a co-vector, and transforms in a dual manner to a vector $\dot{g} \in T_g G$. With reference to Figure \ref{fig:CommDiag_G}, the gradient can be represented as 
\begin{equation}\label{eq:bar_grad_G}
    \bar{\text{d}}_g V := (L_g)^* \text{d}V(g) \in\mathfrak{g}^*\,,
\end{equation}
and as 
\begin{equation}
    \text{d}_g V := {\Lambda^*} (\bar{\text{d}}_g V) \in \mathbb{R}^n\,.
\end{equation}

Equivalently, $\text{d}_g V$ can be found from the computation in a chart $(U_j,X_j)$ as (indeed, dual to \eqref{eq:G_vec_to_algebra_rep})
\begin{equation}\label{eq:G_covec_to_algebra}
    \text{d}_g V = K(q_j)^{-\top}\frac{\partial V}{\partial q_j}_{|q_j = X_j(g)}\,.
\end{equation}
Here, a computationally advantageous choice can be made for the chart map $X_j$: by choosing the chart $(U_j,X_j) = (U_g,X_g)$ in \eqref{eq:one_chart_G} one finds that $K(X_j(g)) = K(0) = I_{n}$, such that the computation of \eqref{eq:dexp_rep} can be avoided:
\begin{align}\label{eq:grad_G}
    \text{d}_g V & = K(q_j)^{-\top}\frac{\partial}{\partial q_j} V\bigg(g\exp\big(\Lambda ({q}_j)\big)\bigg)_{|q_j = 0} \nonumber \\ &= \frac{\partial}{\partial q_j} V\bigg(g\big(I+\Lambda({q}_j)\big)\bigg)_{|q_j = 0} \,. 
\end{align}
The final simplification in equation \eqref{eq:grad_G} holds for matrix Lie groups, where higher order terms of the power series \eqref{eq:general_matrix_exp} can be neglected.

\subsection{Composition of Lie groups}\label{ssec:technical_tools:composition} 
We briefly review the composition of Lie groups. Lie groups $G$ and $H$ can always be composed to form a product Lie group $G\times H$. For matrix Lie groups $G \subset GL(m,\mathbb{R})$ and $H \subset GL(l,\mathbb{R})$, a product Matrix Lie group $G\times H \subset GL(m+l,\mathbb{R})$ can be defined as \citep[Definition 4.17]{Hall2015} 
\begin{equation}
    G\times H := \{\begin{bmatrix} g & 0 \\ 0 & h \end{bmatrix} \, | \, g\in G,\, h\in H\}\,.
\end{equation}

The composition of matrix Lie groups has a block-diagonal structure. This block-diagonal structure reappears in the construction of the corresponding Lie algebra $\mathfrak{g}\oplus\mathfrak{h} \subset gl(m+l,\mathbb{R})$, the adjoint map, 
exponential map 
and the logarithmic map, which 
consist of their counterparts for $G$ and $H$. The algebra representation $\Lambda:\mathbb{R}^{n+k} \rightarrow \mathfrak{g}\oplus\mathfrak{h}$ can likewise be chosen to consist of the components $\Lambda_G:\mathbb{R}^{n} \rightarrow \mathfrak{g}$ and $\Lambda_H:\mathbb{R}^{k} \rightarrow \mathfrak{h}$.

\section{The cases $SO(3)$ and $SE(3)$}\label{sec:kin} 
Prior theory is applied to the Lie groups $SO(3)$ and $SE(3)$.  

\subsection{The matrix Lie groups SO(3) and SE(3)} \label{ssec:kin:prior_def} 

Here, the special orthogonal group $SO(3)$ and the special Euclidean group $SE(3)$ are directly defined as matrix Lie groups that collect transformations of the Euclidean 3-space $\mathbb{R}^3$. $SO(3)$ can be described as the collection of rotations of a vector space $\mathbb{R}^3$, and $SE(3)$ as the collection of simultaneous rotations and translations of $\mathbb{R}^3$, implemented on the vector space of homogeneous vectors (vectors of the form $\begin{pmatrix} x \\ 1 \end{pmatrix}$ with $x\in \mathbb{R}^3$, see \cite[Ch. 3.1]{Murray1994}). 

Define $SO(3) \subset GL(3,\mathbb{R})$ and $SE(3) \subset GL(4,\mathbb{R})$ as the matrix Lie groups
\begin{align} 
    SO(3) &:= \{R \in \mathbb{R}^{3\times3} | \, R^\top R = I,\, \text{det}(R) = 1\}\,, \\
    SE(3) &:= \{\begin{bmatrix} R & p \\ 0& 1 \end{bmatrix} \in \mathbb{R}^{4\times4}| \, R\in SO(3),\,p \in \mathbb{R}^3\} \,,
\end{align}
in both cases using matrix composition as the group operation. 

\begin{remark}\label{remark:H^A_B}
    Concerning notation for relative poses of rigid bodies: $H^A_B \in SE(3)$ indicates the pose of a reference frame $\Psi_B$ as seen from $\Psi_A$, while $H^B_A = {H^A_B}^{-1}$. 
\end{remark}

The Lie algebras of $SO(3)$ and $SE(3)$ are the vector spaces $so(3) \subset gl(3,\mathbb{R})$ and $se(3)\subset gl(4,\mathbb{R})$, respectively, with their Lie bracket given by the matrix commutator. 
The Lie algebras $so(3)$ and $se(3)$ are given by 
\begin{align}
    so(3) &:= \{\Tilde{\omega} \in \mathbb{R}^{3\times3} \, | \, \Tilde{\omega} = - \Tilde{\omega}^\top\}\,, \\
    se(3) &:= \{\begin{bmatrix} \Tilde{\omega} & v \\ 0& 0 \end{bmatrix} \in \mathbb{R}^{4\times4} \, | \, \Tilde{\omega}\in so(3), v \in \mathbb{R}^3\}\,.
\end{align}
Arbitrary elements of $so(3)$ and $se(3)$ will be denoted by $\tilde{\omega}$ and $\tilde{T}$, respectively.

The vector space isomorphism $\Lambda_{SO(3)}:\mathbb{R}^3\rightarrow so(3)$ is defined as
\begin{equation}
    \Lambda_{SO(3)}(\omega) = \Lambda_{SO(3)}\begin{pmatrix} \omega_1 \\ \omega_2 \\ \omega_3 \end{pmatrix} 
    := 
    \begin{bmatrix} 
    0 & -\omega_3 & \omega_2 \\ 
    \omega_3 & 0 & -\omega_1 \\
    -\omega_2 & \omega_1 & 0
    \end{bmatrix}\,. 
\end{equation}
For $SE(3)$, define 
$\Lambda_{SE(3)}:\mathbb{R}^6 \rightarrow se(3)$ for $T = \begin{pmatrix} \omega \\ v \end{pmatrix} \in \mathbb{R}^6$ with $\omega,v \in \mathbb{R}^3$, via
\begin{equation}
    \Lambda_{SE(3)}(T) := 
    \begin{bmatrix} \Lambda_{SO(3)} (\omega) & v \\ 0& 0 \end{bmatrix}\,.
\end{equation}
Both $\Lambda_{SO(3)}$ and $\Lambda_{SE(3)}$ will be denoted as $\Lambda$ in the following, since it is clear from context which one is meant.

\begin{remark}\label{remark:T^{A,B}_C}
    Concerning notation for the relative twists (velocities) of rigid bodies: Consider a curve $H^A_B:\mathbb{R}\rightarrow SE(3)$, then twists $\Tilde{T} \in se(3)$ appear as the left and right translated change-rates $\dot{H}^A_B = \frac{d}{dt} H^A_B(t)$ by
\begin{equation} \label{eq:infinitesimal_twists}
    \Tilde{T}^{B,A}_B = H^B_A \dot{H}^A_B \,, \qquad    \Tilde{T}^{A,A}_B = \dot{H}^A_B H^B_A \,.
\end{equation}
Both $\Tilde{T}^{B,A}_B$ and $\Tilde{T}^{A,A}_B$ represent the generalized velocity (twist) of frame $\Psi_B$ with respect to $\Psi_A$, but $\Tilde{T}^{B,A}_B$ is expressed in $\Psi_B$ while $\Tilde{T}^{A,A}_B$ is expressed in $\Psi_A$ \citep[ Ch. 2.4]{Murray1994}.
\end{remark}

The adjoint representations 
of $so(3)$ and $se(3)$ follow from the definition 
\eqref{eq:algebra_adjoint} as 
\begin{align}
     \text{ad}_\omega = \Lambda(\omega) \,, \quad
     \text{ad}_T = \begin{bmatrix} \Lambda(\omega) & 0 \\ \Lambda(v)& \Lambda(\omega) \end{bmatrix}\,.
\end{align}

The exponential maps for $SO(3)$ and $SE(3)$ are almost-global diffeomorphisms that relate $\Tilde{\omega} \in so(3)$ to $R \in SO(3)$ via \eqref{eq:expSO3} and $\Tilde{T}\in se(3)$ to $H \in SE(3)$ via \eqref{eq:expSE3} \citep[App. A, Sec. 2.3]{Murray1994}:
\begin{align}
    e^{\Tilde{\omega}} = \sum_{n = 0}^{\infty} \frac{1}{n!} \Tilde{\omega}^n = I + \sin(\theta) \Tilde{\hat{\omega}} + (1-\cos(\theta))\Tilde{\hat{\omega}}^2 \,, \label{eq:expSO3} \\
    e^{\Tilde{T}} = \sum_{n = 0}^{\infty} \frac{1}{n!} \Tilde{T}^n = 
    \begin{bmatrix} e^{\Tilde{\omega}} & \frac{1}{\theta^2}(I-e^{\Tilde{\omega}}){\Tilde{\omega}}v+\omega^{\top}v\omega \\ 0& 1 \end{bmatrix} \,, \label{eq:expSE3} 
\end{align}
with $\theta = \|\omega\|_2 = \sqrt{\omega^\top\omega}$ and $\Tilde{\hat{\omega}} = \Tilde{\omega}/\theta$. 

For $\theta < \pi$ their inverses are presented in equations \eqref{eq:log_SO3} and \eqref{eq:log_SE3}, respectively: 
the log map for $SO(3)$ is \citep[App. A, Sec. 2.3]{Murray1994}
\begin{equation}\label{eq:log_SO3}
    \log(R) = 
    \begin{cases}
    \cos^{-1}(\frac{1}{2}(\text{Tr}(R)-1))\frac{A}{\|A\|}\,, \qquad &R \neq I \,, \\
    0_{3\times3}\,, \qquad &  R = I \,,
    \end{cases}
\end{equation}
with $A = \frac{1}{2}(R-R^\top)$ the anti-symmetric part of $R$, while $\|A\|:= \sqrt{-\frac{1}{2}\text{Tr}(A^2)}$. \\
Denoting $\Tilde{\omega} = \text{log}(R)$, the log map for $SE(3)$ is (\cite{Murray1994}, App. A, Sec. 2.3)
\begin{align}
    &\log(\begin{bmatrix} R & p \\ 0& 1 \end{bmatrix}) = \begin{bmatrix} \Tilde{\omega} & Q p \\ 0& 0  \end{bmatrix}\,, \label{eq:log_SE3} \\
    &Q = I - \frac{1}{2} \Tilde{\omega} + \frac{2\sin(\theta)-\theta(1+\cos(\theta))}{2\theta^2\sin(\theta)}\Tilde{\omega}^2 \,. \label{eq:Q_for_SE3_log} 
\end{align}
Since $\lim_{\theta \rightarrow 0} Q = I$, a well-defined $Q$ is given by \eqref{eq:Q_for_SE3_log} regardless of $R$, such that the logarithm on $SE(3)$ \eqref{eq:log_SE3} has the range of validity of the logarithm on $SO(3)$ \eqref{eq:log_SO3}, bounded only by the rotational part. 


\subsection{Minimal atlas}\label{ssec:kin:MinAt}
Here we construct minimal exponential atlases for $SO(3)$ and $SE(3)$ as special cases of the exponential atlas \eqref{eq:one_chart_G} - \eqref{eq:one_chart_G_X_inv} for the respective Lie groups. Both atlases use four charts, which is the theoretical minimum size of an atlas for $SO(3)$ and $SE(3)$ \citep{Grafarend2011}.

For the atlas on $SO(3)$ the four exponential charts are centered on the elements 
\begin{align}
    & R_0 = 
    \begin{bmatrix*}[r] 1 & 0 & 0 \\ 0 & \phantom{-}1 & 0 \\ 0 & 0 & \phantom{-}1 \end{bmatrix*}
    \,, \quad 
    R_1 = 
    \begin{bmatrix*}[r] 1 & 0 & 0 \\ 0 & -1 & 0 \\ 0 & 0 & -1 \end{bmatrix*} 
    \,, \nonumber \\
    & R_2 = 
    \begin{bmatrix*}[r] -1 & 0 & 0 \\ 0 & \phantom{-}1 & 0 \\ 0 & 0 & -1 \end{bmatrix*}
    \,, \quad 
    R_3 = 
    \begin{bmatrix*}[r] -1 & 0 & 0 \\ 0 & -1 & 0 \\ 0 & 0 & \phantom{-}1 \end{bmatrix*}
    \nonumber \,.
\end{align}
The full minimal atlas for $SO(3)$ is then given by 
\begin{align}\label{eq:MinAtlas_SO3}
    \mathcal{A}_{\min}^{SO(3)} &:= & \{(U_j,x_j) \;|\; \, j \in \{0,1,2,3\}\, \}&\,, \\
    U_j &:= & \{R_j e^{\Tilde{\omega}} \;|\; \omega \in \mathbb{R}^3, |\omega|<\pi  \}&\,, \label{eq:chart_region_SO3} \\
    x_j(R) &:= & \Lambda^{-1}\text{log}(R_j^\top R)&\,, \label{eq:chart_map_SO(3)} \\
    x_j^{-1}(q_j) &= & R_j e^{\Lambda(q_j)}&\,.
\end{align}
Intuitively speaking, the open set $U_{j}$ contains all orientations that are reachable from $R_j$ by a rotation through an angle less than $\pi$. 

A proof that $\mathcal{A}_{\min}^{SO(3)}$ covers $SO(3)$ is shown in Appendix \ref{ssec:app_MinAt_Complete}.


\noindent
For the atlas on $SE(3)$, define the centers of the exponential charts on $SE(3)$ as
\begin{equation*}
    H_j := \begin{bmatrix}R_j & 0 \\ 0 & 1 \end{bmatrix}\,, \quad j \in \{0,1,2,3\}\,.
\end{equation*}
The full minimal atlas on $SE(3)$ is then given by

\begin{align}\label{eq:MinAtlas_SE3}
    \mathcal{A}_{\min}^{SE(3)} &:=& \{(\mathcal{U}_j,\mathcal{X}_j) \;|\; (U_j,x_j) \in \mathcal{A}_{\min}^{SO(3)}\}&\,, \\
    \mathcal{U}_j &:=& \{\begin{bmatrix} R & p \\ 0& 1 \end{bmatrix} \;|\; R \in U_j, p \in \mathbb{R}^3\}&\,, \label{eq:chart_region_SE3} \\
    \mathcal{X}_j(H) &:=& \Lambda^{-1}\text{log}(H_j^{-1}H)&\,, \\
    \mathcal{X}_j^{-1}(q_j) &=& H_j e^{\Lambda(q_j)}&\,.
\end{align}

\subsection{Expressing scalar functions}\label{ssec:kin:FunDef}

\begin{figure}
 \centering
 \includegraphics[height=10em]{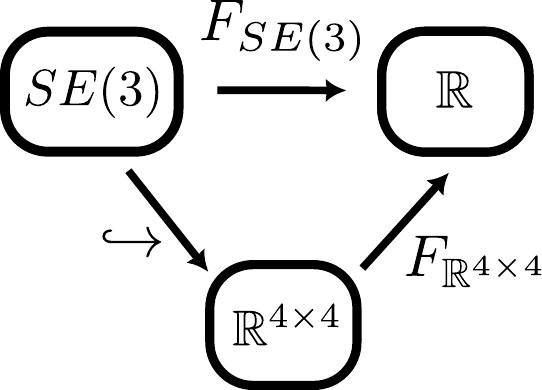}
 \caption{Commutative diagram highlighting how the natural embedding $\hookrightarrow:SE(3)\rightarrow\mathbb{R}^{4\times 4}$ can be used to restrict a general matrix function to $SE(3)$. }\label{fig:Function_SE3}
\end{figure}

We briefly highlight how to represent scalar-valued functions $F_{SO(3)}:SO(3)\rightarrow\mathbb{R}$ and $F_{SE(3)}:SE(3)\rightarrow\mathbb{R}$. One approach is to define functions on manifolds by restricting a function defined on an embedding space $\mathbb{R}^{n\times n}$ \citep{Falorsi2020}. For Lie groups, it is immediately applicable whenever a matrix representation is available. 

For example, on $SE(3)$ one restricts a function $F_{\mathbb{R}^{4\times 4}}:\mathbb{R}^{4\times4}\rightarrow\mathbb{R}$ to arguments from $SE(3)\subset \mathbb{R}^{4\times 4}$, see Figure \ref{fig:Function_SE3}. 
The gradient of such functions are computed by an application of equation \eqref{eq:grad_G}:

\begin{equation}\label{eq:grad_SE3}
    \text{d}_H F_{SE(3)} = \frac{\partial}{\partial q} F_{SE(3)}\big(H(I+\tilde{q})\big)_{q=0} \,.
\end{equation}

The approach also holds for $SO(3)$, which can be embedded in $\mathbb{R}^{3\times 3}$, instead of $\mathbb{R}^{4\times 4}$.


\section{Optimizing a rigid body control} \label{sec:opt} 

The optimization procedure in Section \ref{sec:main_result} is applied to potential energy and damping injection based control of a fully actuated rigid body. 

The core idea of potential energy shaping and damping injection is to combine advantages of energy-balancing passivity based control (EB-PBC) \citep{Ortega2000} and of control by interconnection \citep{Ortega2008}, which provide stability guarantees when interfacing with physical systems. Our article presents a class of controllers that generalizes the architecture presented by \cite{Rashad2019}. We address common safety concerns about machine learning in control loops by optimizing a class of controllers that guarantees stability and a bounded energy by design.

\subsection{Control of a rigid body}\label{ssec:dyn_ctrl}
 
The trajectory of a rigid body in Euclidean 3-space is fully described by the curve  $H^0_b:\mathbb{R}\rightarrow SE(3)$ that gives the relative position and orientation of a frame $\Psi_b$ attached to the rigid body with respect to an inertial frame $\Psi_0$ (see Remark \ref{remark:H^A_B}). Following equation \eqref{eq:infinitesimal_twists}, the twist of the body with respect to $\Psi_0$, expressed in the body frame $\Psi_b$ is
\begin{equation}
    T^{b,0}_b = \Lambda^{-1} (H^b_0\dot{H}^0_b)\,.
\end{equation}
Given the inertia tensor $\mathcal{I}\in\mathbb{R}^{6\times 6}$, $P^b = \mathcal{I}T^{b,0}_b \in\mathbb{R}^6$ represents the momentum in the body frame. The indices remain fixed in the subsequent treatment, and are suppressed to avoid cluttering (i.e., $H := H^0_b,\, T := T^{b,0}_b,\, P := P^b)$.

The dynamics of a rigid body follow from the Hamiltonian equations on matrix Lie groups (\eqref{eq:Lie_Ham_Matrix} and \eqref{eq:Lie_Ham_Co_Matrix} in Appendix \ref{app:background_Hamiltonian_systems_on_Lie_groups}) by setting $G = SE(3)$ and letting $\mathcal{H}(H,P)= \frac{1}{2}{P}^\top\mathcal{I}^{-1}P$. Including an external input wrench $W\in \mathbb{R}^6$, the dynamics read:
\begin{align}
    \dot{H} &= H \Lambda (\mathcal{I}^{-1}P)\,, \label{eq:Kin_SE3} \\
    \dot{P} &= \text{ad}_{\mathcal{I}^{-1}P}^\top P + W\,. \label{eq:DynSE3}
\end{align}

In control by potential energy shaping and damping injection, this external wrench is constructed as a sum of a potential gradient term $W_V$ and a damping term $W_D$:
\begin{equation} \label{eq:ext_wrench}
    W = W_V + W_D\,.
\end{equation}
In our approach the potential gradient term 
\begin{equation}
    W_V = - \text{d}_H V
\end{equation} is computed by an application of equation \eqref{eq:grad_SE3}.

Nonlinear, configuration-dependent viscous damping takes the form
\begin{equation}
    W_{D} = -B(H, P) P\,,
\end{equation}
with $B(H, P)\mathcal{I} \in \mathbb{R}^{6\times 6}$ a symmetric and positive definite matrix. \endnote{Symmetry and positive definiteness of $B(H,P)$ are not well defined because it is a $(1,1)$ tensor: to be technically precise one must impose that the $(0,2)$ tensor $B(H, P)\mathcal{I}$ is symmetric and positive definite. This imposes that the rate of energy lost to damping $T^\top B \mathcal{I} T \geq 0$. } 

\begin{remark}\label{remark:quadratic_control}In this context, the control architecture of \cite{Rashad2019} corresponds to the popular yet very particular choice of a constant $B(H,P)$ for the damping injection, while their potential $V(H)$ shows a quadratic dependence on translations and a nearly quadratic dependence on rotations. Their controller may be interpreted as a linear PD controller on $SE(3)$, where our work may be seen as a nonlinear PD controller on $SE(3)$.
\end{remark}

\subsection{Stability}\label{ssec:Stability}
We present here a general proof of stability for the class of controllers. 
\begin{theorem}[Stability]\label{thm:stability}
    Given the system \eqref{eq:Kin_SE3}, \eqref{eq:DynSE3} together with the controller \eqref{eq:ext_wrench} given as 
    \begin{equation}
        W(H,P) = - \text{d}_{H} V -B(H, P) P\,.
    \end{equation}
    If $\forall H\in SE(3), P\in \mathbb{R}^6: B(H, P)\mathcal{I} > 0$ and $V:SE(3)\rightarrow \mathbb{R}$ is lower-bounded, then $\lim_{t\rightarrow \infty} P(t) = 0$ and any local minimum of $V$ is locally asymptotically stable.
\end{theorem}

\begin{proof}
    With $\text{E}_{\text{Pot}} = V(H)$ and $\text{E}_{\text{Kin}} = \frac{1}{2} {P}^\top\mathcal{I}^{-1} P$, take the system's total energy $\text{E} = \text{E}_{\text{Pot}} + \text{E}_{\text{Kin}}$ as the Lyapunov function candidate. Then 
    \begin{align}
        \dot{\text{E}}_{\text{Pot}} &= (\text{d}_{H}V)^\top T \\
        \dot{\text{E}}_{\text{Kin}} &= (\dot{P})^\top\mathcal{I}^{-1} P = (\dot{P})^\top T \\ 
         &=  (\text{ad}^\top_{T}P - \text{d}_{H} V)^\top T - (P)^\top B T \nonumber \\
         &= (-\text{d}_{H}V)^\top T - T^\top B\mathcal{I} T  \,, \nonumber
    \end{align}
    such that  
    \begin{equation}
        \dot{\text{E}} = \dot{\text{E}}_{\text{Pot}} + \dot{\text{E}}_{\text{Kin}} = -T^\top B\mathcal{I} T \leq 0 \,.
    \end{equation}
    By LaSalle's invariance principle, the system converges to the greatest invariant subset where $\dot{E} = 0$. Since $B(H, P) > 0$, the set with $\dot{E} = 0$ is simply the set 
    \begin{equation}
        S = \{(H,P)\, |\, P = 0 \}\,.
    \end{equation}
    By inspection of the dynamics \eqref{eq:DynSE3} the greatest invariant subset of $S$ is the set with $\text{d}_{H} V = 0$, corresponding to the maxima and minima of $V$. Since $\dot{E} \leq 0$, the system cannot converge to maxima of $V$, leaving the minima of $V$ as limit sets and local minima of $V$ as asymptotically stable equilibria.
    \qed
\end{proof}



\subsection{Optimization by the adjoint method on SE(3)}\label{ssec:opt_SE3}

In order to apply the adjoint method on Lie groups (Theorem \ref{thm:adjoint_method_matrix_lie_group}), we re-define the dynamics \eqref{eq:Kin_SE3} and \eqref{eq:DynSE3} on the Lie group $G = SE(3) \times se^*(3)$. To this end, we choose to equip $G$ with the element-wise composition $(H_1,P_1)\circ(H_2,P_2) = (H_1H_2,P_1+P_2)$. 

As a composition of matrix Lie groups $SE(3)$ and $(\mathbb{R}^6,+)$, $G \subset GL(11,\mathbb{R})$ is defined as
\begin{equation}
    G := \{ \begin{bmatrix} H & 0 & 0 \\ 0 & I & P \\ 0& 0 & 1  \end{bmatrix} \, | \, H \in SE(3), P \in \mathbb{R}^6 \}\,,
\end{equation}
where matrix multiplication indeed corresponds to the element-wise composition in the abstract $G$. 
For details on the construction of $G$, the Lie algebra $\mathfrak{g}$, the choice of Lie algebra representation $\Lambda:\mathbb{R}^{12}\rightarrow\mathfrak{g}$, adjoint map and exponential map see Appendix \ref{app:Lie_group_SE3_x_R^6}.
%
%
%
The dynamics for $\Gamma(t) \in G$ read 
\begin{align}\label{eq:dyn_SE3_full}
    \dot{\Gamma} & = f_\theta(\Gamma) = \begin{bmatrix} \dot{H}(H,P) & 0 & 0 \\ 0 & 0 & \dot{P}_\theta(H,P) \\ 0& 0 & 0  \end{bmatrix} \,,  
\end{align}
where 
\begin{align}
    \dot{H}(H,P) &= H \Lambda (\mathcal{I}^{-1}P) \,,\\
    \dot{P}_\theta(H,P) &= \text{ad}_{\mathcal{I}^{-1}P}^\top P + W_\theta(H,P) \label{eq:dyn_SE3_dPdt}\,.
\end{align} 
Here, the control-wrench $W_\theta:SE(3)\times\mathbb{R}^6 \rightarrow \mathbb{R}^6$ is parameterized by $\theta \in \mathbb{R}^{n_\theta}$.

Given a cost $C^T_{f_{\theta}}(\Gamma,\theta)$ of the type \eqref{eq:single_trajectory_cost} and a distribution $\mathbb{P}$ of initial conditions $\Gamma_0$, define an optimization problem in the form of \eqref{eq:abstract_optimization_problem}:
\begin{equation}\label{eq:optimization_constraints}
    \min_\theta \, \mathbb{E}_{\Gamma_{0} \sim \mathbb{P}(\Gamma_{0})}[C(\Gamma_0,\theta)] = \min_\theta J(\theta) \,. 
\end{equation}

As in Section \ref{ssec:main_result/main_result}, approximate $J(\theta) \approx \sum_{i=0}^N C^T_{f_\theta}(\Gamma_i,\theta)$ and apply Theorem \ref{thm:adjoint_method_matrix_lie_group} to compute the parameter gradient of the $C^T_{f_\theta}(\Gamma_i,\theta)$ by equation \eqref{eq:parameter_gradient_matrix_lie_group}.

The dynamics of $\lambda_\Gamma$ follow from equation \eqref{eq:Mat_Lie_Adjoint_Sensitivity_Co} as
\begin{align}\label{eq:adjoint_sensitvity_co_se3}
    & \dot{\lambda}_\Gamma = - \text{d}_\Gamma\bigg(\lambda_\Gamma\big(\Lambda^{-1}(\tilde{f}_\theta(\Gamma)\big)+r(\Gamma,\theta)\bigg)  
        + \text{ad}_{\tilde{f}_\theta}^\top \lambda_\Gamma \,, \\
    & \lambda_{\Gamma}(T) = \text{d}_{\Gamma} F\,,
\end{align}
where $\Lambda^{-1}(\tilde{f}_\theta)$ is
\begin{align}
    \Lambda^{-1}(\tilde{f}_\theta) &:= \Lambda^{-1}\big(\Gamma^{-1}f_{\theta}(\Gamma)\big) = \begin{pmatrix} \mathcal{I}^{-1}P \\ \dot{P}_\theta(H,P) \end{pmatrix}  \,.
\end{align}

The gradient $\text{d}_\Gamma:C^1(G,\mathbb{R})\rightarrow\mathbb{R}^{12}$ is given by Equation \eqref{eq:grad_G}. We split it into components $\text{d}_H:C^1(SE(3),\mathbb{R})\rightarrow\mathbb{R}^6$ and $\text{d}_P:C^1(\text{Vec}(6,\mathbb{R}),\mathbb{R})\rightarrow\mathbb{R}^6$ defined on the Lie groups $SE(3)$ and $(\mathbb{R}^6,+)$, respectively.

Further, split $\lambda_\Gamma= (\lambda_H, \lambda_P)$ into components $\lambda_H, \lambda_P \in \mathbb{R}^6$, and write out the control wrench $W_\theta(H,P) = - \text{d}_H V_\theta - B_\theta(H,P)P$.
Then the equation for $\dot{\lambda}_\Gamma$ can be resolved to 
\begin{align}\label{eq:co-state_Full}
    \dot{\lambda}&_\Gamma =  \\ 
    &
    \begin{pmatrix} 
    \text{d}_H \big(\lambda_P^\top (\text{d}_HV_\theta + B_\theta(H,P)P) - r(\Gamma,\theta)\big) 
    - \text{ad}_{\mathcal{I}^{-1}P}^\top \lambda_H 
    \\  
    \text{d}_P\big(\lambda_P^\top( B_\theta(H,P)P - \text{ad}_{\mathcal{I}^{-1}P}^\top P) - r(\Gamma,\theta)\big) 
    - \mathcal{I}^{-1}\lambda_H 
    \end{pmatrix} \nonumber \,, \\
    \Lambda&_H(T) = \text{d}_H F \,,\; \Lambda_P(T) = \text{d}_P F\,. \nonumber
\end{align}
\begin{remark}
    Note that the second derivative term $\text{d}_H \big(\lambda_P^\top (\text{d}_H V_\theta)\big)$ in the dynamics \eqref{eq:co-state_Full} is well defined, since $\lambda_P^\top (\text{d}_H V_\theta) \in C^1(SE(3),\mathbb{R})$.
\end{remark}

\begin{remark}\label{remark:group_action_se*3}
    
    Rather than constructing the direct product group $SE(3)\times se^*(3)$, the semi-direct product group $SE(3)\ltimes se^*(3)$ could have been defined using the Coadjoint representation $Ad_H^*:se^*(3)\rightarrow se^*(3)$, leading to the group operation $(H_1,P_1)\bullet (H_2,P_2) = (H_1H_2, P_1 + Ad_{H_1}^* P_2) $. Since an alternative choice of group action does not affect the optimum of the optimization, use of the semi-direct product group was not further investigated.  
    
\end{remark}


\section{Simulations}\label{sec:train}

We numerically solve the optimization problem \eqref{eq:abstract_optimization_problem} for the dynamics \eqref{eq:dyn_SE3_full}. We investigate various choices of final and running costs, distributions $\mathbb{P}$ and parameterizations of $V_\theta$, $B_\theta$.

\subsection{Quadratic vs. general potential shaping}\label{ssec:ex}
A controller with quadratic potential and linear damping injection (Section \ref{ssec:train:quadratic_training}) is compared to a controller with NN-parameterized potential and damping injection (Section \ref{ssec:train:nonlinear_training}).

\subsubsection{Choice of cost $C$ and distribution $\mathbb{P}$:} \label{ssec:train:cost_1}
We determine a final cost $F$ and a running cost $r$ to stabilize a static target state with $H = H_F,\,P=0$ over a horizon of $T = 3$ seconds.
The key properties of $F$ and $r$ are that both are differentiable and have their minimum in the target pose. Denote components of $H$ and $P$
\begin{equation*}
    H = \begin{bmatrix} R & p \\ 0 & 1\end{bmatrix}\,,\; P = \begin{pmatrix} P_\omega \\ P_v \end{pmatrix}\,,
\end{equation*}
where $R\in \mathbb{R}^{3\times 3}$, and $p,P_\omega,P_v \in \mathbb{R}^3$. With weights $w_1,\hdots,w_9 \in \mathbb{R}_+$, 
we choose $F$ and $r$ as
\begin{align}
    F(\Gamma) =& -w_1\text{Tr}(H_F^{-1}H) + w_2 \| p \|_2^2 \label{eq:specific_final_cost} \\ &+ w_3 \|P_\omega\|^2 + w_4 \|P_v\|_2^2 \nonumber \,, \\
    r(\Gamma) =&  - w_{5}\text{Tr}(H_F^{-1}H) + w_{6} \| p \|_2^2 \label{eq:specific_running_cost} \\ & + w_{7} \|P_\omega\|^2 + w_{8} \|P_v\|_2^2 + w_9 \|W_\theta(H,P)\|^2 \nonumber \,.
\end{align} 
 
Given scalars $\alpha, d, \theta_p, d_p \in \mathbb{R}$ and vectors $\omega, v, \omega_p, v_p \in \mathbb{R}^3$, and an average initial pose $H_I = H_F$, an initial condition $\Gamma_0$ is constructed as 
\begin{align}\label{eq: prior_dist}
    &\Gamma_0 = \begin{pmatrix}
    H \\ P
    \end{pmatrix} \\
    &q = \begin{pmatrix} \alpha \omega/ \|\omega\|_2 \\ d v/ \|v|_2 \end{pmatrix} \,, \\
    &H = H_I\text{exp}(\tilde{q}) \,, \\
    &P = \begin{pmatrix} \alpha_p \omega_p/ \|\omega_p\|_2 \\ d_p v_p/ \|v_p|_2 \end{pmatrix} \,.
\end{align}
The distribution $\mathbb{P}$ of $\Gamma_0$ is implemented by sampling 
\begin{align} \label{eq: uniform_sample_prior}
    &\alpha \sim \text{Uniform}[0,\pi]\,, \quad
    d \sim \text{Uniform} [0,1] \,,\\
    &\alpha_p \sim \text{Uniform}[0, 0.03]\,, \quad
    d_p \sim \text{Uniform}[0, 1] \,,
\end{align}
and sampling $\omega,v, \omega_p, v_p \in \mathbb{R}^3$ from a normal distribution $\mathcal{N}(\mu, \sigma^2)$ with standard deviation $\mu = (0,0,0)^T$ and variance $\sigma^2 = I_{3\times3}$. 

\subsubsection{Quadratic potential and linear damping injection:}\label{ssec:train:quadratic_training}
The quadratic controller coincides with the controller presented by \cite{Rashad2019}, in a setting of motion control. As such the quadratic potential $V_{Q,\theta}:SE(3)\rightarrow \mathbb{R}$ is given by
\begin{align}
    V_{Q,\theta}(\begin{bmatrix} R & p \\ 0& 1 \end{bmatrix}) 
    = \, & \frac{1}{4}p^\top K_\theta p + \frac{1}{4}p^\top R K_\theta R^\top p \\
    & - \text{Tr}(G_\theta (R - I_{3}))\,, \nonumber
\end{align}
and the constant damping injection is characterized by
\begin{equation}
    B_{Q,\theta}(H,P) = B_{Q,\theta} \in \mathbb{R}^{6\times6}\,.
\end{equation}
Here the translational stiffness matrix $K_\theta\in\mathbb{R}^{3\times 3}$, the rotational co-stiffness matrix $G_\theta\in\mathbb{R}^{3\times 3}$ and the damping injection matrix $B_{Q,\theta}$ are chosen as 
\begin{align}
    K_\theta & = 
    \begingroup
    \setlength\arraycolsep{0pt}
    \begin{bmatrix}
        \exp(\theta_1) & 0 & 0 \\
        0 & \exp(\theta_2) & 0 \\
        0 & 0 & \exp(\theta_3) \\
    \end{bmatrix} 
    \endgroup \\
    G_\theta & = 
    \begingroup
    \setlength\arraycolsep{0pt}
    \begin{bmatrix}
        \exp(\theta_4) & 0 & 0 \\
        0 & \exp(\theta_5) & 0 \\
        0 & 0 & \exp(\theta_6) \\
    \end{bmatrix}  
    \endgroup \\
    B_{Q,\theta} & = 
    \begingroup
    \setlength\arraycolsep{-3pt}
    \begin{bmatrix}
        \exp(\theta_7) & 0 & 0 & 0 & 0 & 0 \\
        0 & \exp(\theta_8) & 0 & 0 & 0 & 0 \\
        0 & 0 & \exp(\theta_9) & 0 & 0 & 0 \\
        0 & 0 & 0 & \exp(\theta_{10}) & 0 & 0 \\
        0 & 0 & 0 & 0 & \exp(\theta_{11}) & 0 \\
        0 & 0 & 0 & 0 & 0 & \exp(\theta_{12}) \\
    \end{bmatrix}
    \endgroup
\end{align}
where the diagonal elements $\exp(\theta_i)$ ensure that the matrices are positive definite. Note that the conditions of Theorem \ref{thm:stability} are guaranteed: $V_Q$ is lower bounded, and symmetry of $B_{Q,\theta}\mathcal{I}$ is guaranteed since $B_{Q,\theta}$ is diagonal and positive definite.

The control-law is then of the form
\begin{equation}\label{eq:quadratic-control-law}
    W_{Q,\theta}(H,P) = -\text{d}_{H}V_{Q,\theta} - B_{Q,\theta}P\,.
\end{equation}


The parameters are optimized over $1200$ training epochs, using the ADAM optimizer with decay $\gamma = 0.999$, a learning rate of $\eta = 0.001$ for the initial $1000$ epochs and restarting training at a learning rate of $\eta = 0.01$ for the final $200$ epochs. Additional parameters of the training are summarized in Appendix \ref{ssec:appB:Hyperparameters}, Table \ref{tab:appB:Hyperparameters_quadratic_training}. The training progress is summarized in Figure \ref{fig:training_progress_quadratic}, where Figure \ref{sfig:quadratic_loss} shows a monotonous decrease of the cost function over the training epochs, while Figures \ref{sfig:quadratic_th_final} and \ref{sfig:quadratic_d_final} indicate a steady improvement of the final system states with respect to the target pose. The resulting controller's performance is shown in Figure \ref{fig:examples_quadratic}. Here Figures \ref{sfig:quadratic_angle_3s} and \ref{sfig:quadratic_dist_3s} show that the controlled rigid bodies approach the target configuration. Figure \ref{sfig:quadratic_energy_components_average} shows that the controller uses the potential $V_{Q,\theta}$ to guide the rigid bodies towards the target pose, and that kinetic energy is quickly dissipated.  

\begin{figure*}[h]
\centering
\begin{tabular}{c c c}
  \begin{subfigure}{.33\textwidth}
    \scriptsize\includegraphics[width = \textwidth]{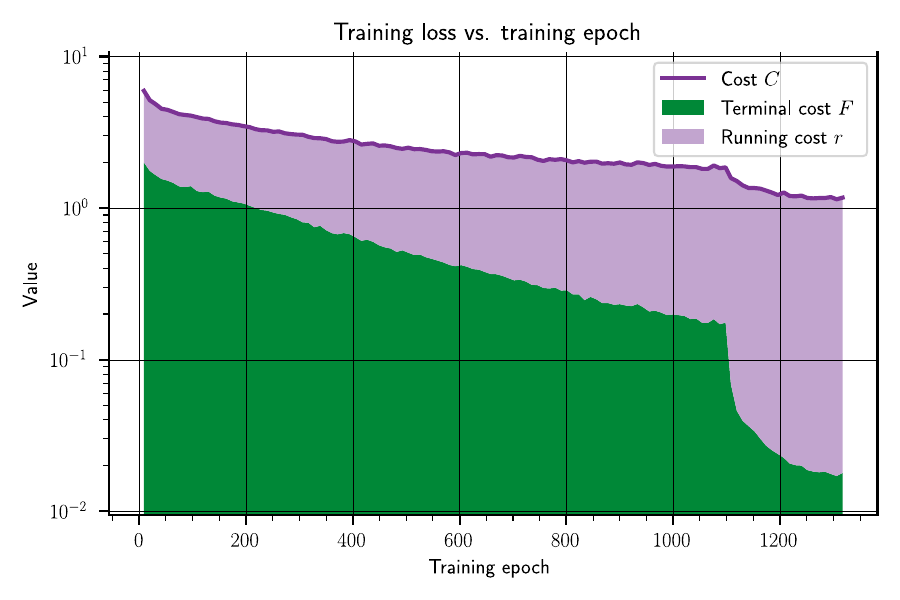}
    \subcaption{Loss, terminal loss and integral loss.}\label{sfig:quadratic_loss}
  \end{subfigure}
  &
  \begin{subfigure}{.33\textwidth}
    \scriptsize\includegraphics[width = \textwidth]{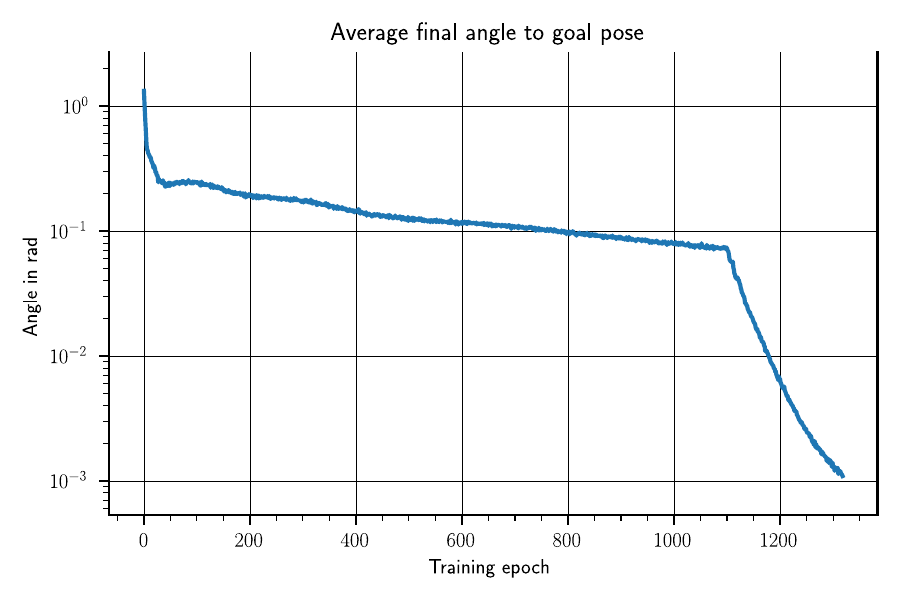}
    \subcaption{Average final angle towards goal pose.}\label{sfig:quadratic_th_final}
  \end{subfigure}
  &
  \begin{subfigure}{.33\textwidth}
    \scriptsize\includegraphics[width = \textwidth]{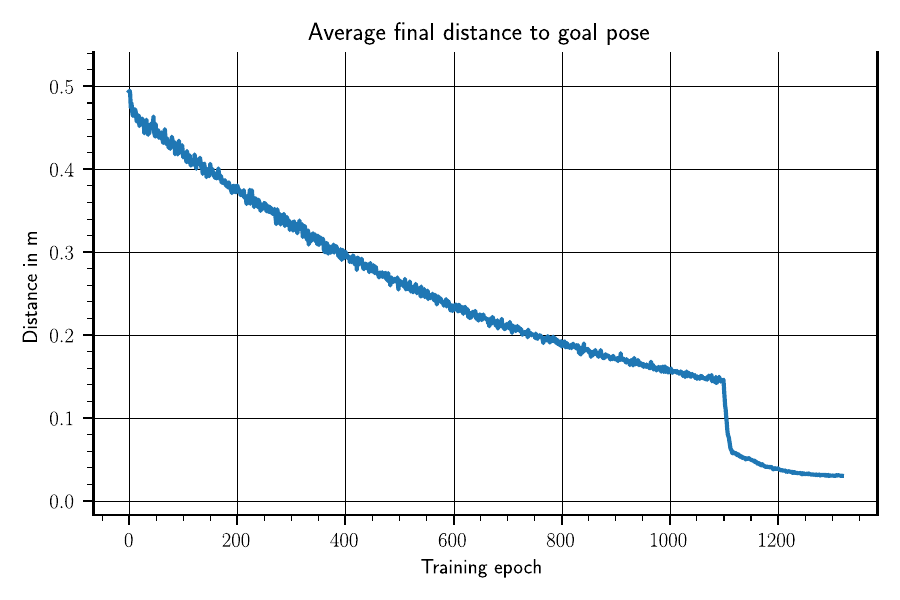}
    \subcaption{Average final distance towards goal pose.}\label{sfig:quadratic_d_final}
  \end{subfigure}
\end{tabular}
    \caption{Visualization of the training progress of the quadratic controller characterized by $V_{Q,\theta}$ and $B_{Q,\theta}$. All figures show data averaged over 2048 sample trajectories at the given epoch, with initial conditions sampled from $\mathbb{P}(\Gamma_0)$.}\label{fig:training_progress_quadratic}
\end{figure*}

\begin{figure*}[h]
\centering
\begin{tabular}{c c c}
  \begin{subfigure}{.33\textwidth}
    \scriptsize\includegraphics[width = \textwidth]{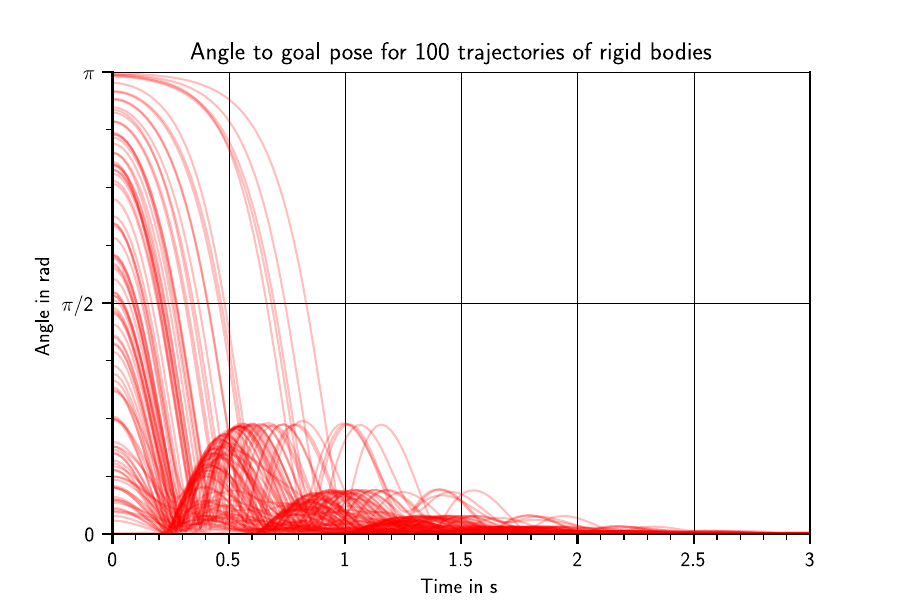}
    \subcaption{Angle towards goal pose over 3 seconds.}\label{sfig:quadratic_angle_3s}
  \end{subfigure}
  &
  \begin{subfigure}{.33\textwidth}
    \scriptsize\includegraphics[width = \textwidth]{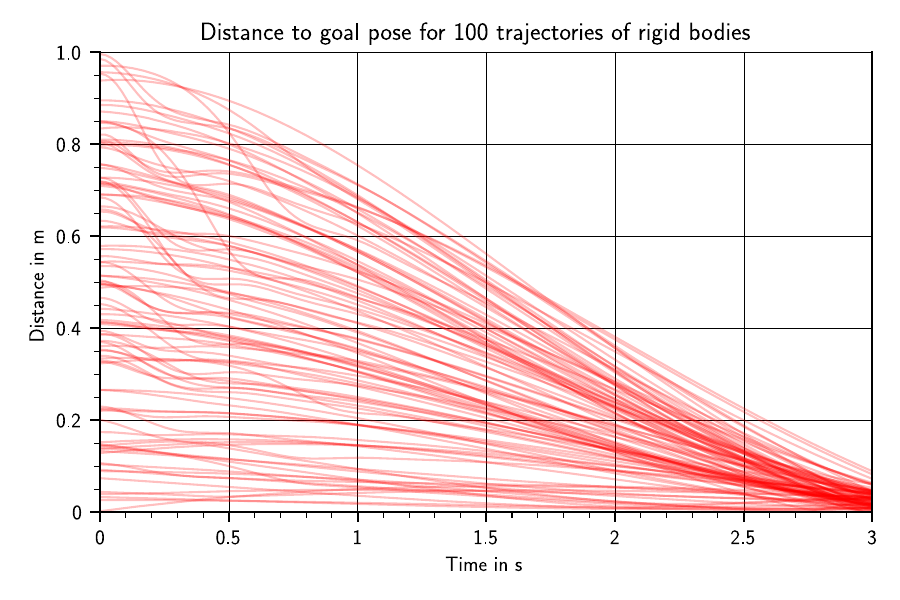}
    \subcaption{Distance towards goal pose over 3 seconds.}\label{sfig:quadratic_dist_3s}
  \end{subfigure}
  &
  \begin{subfigure}{.33\textwidth}
    \scriptsize\includegraphics[width = \textwidth]{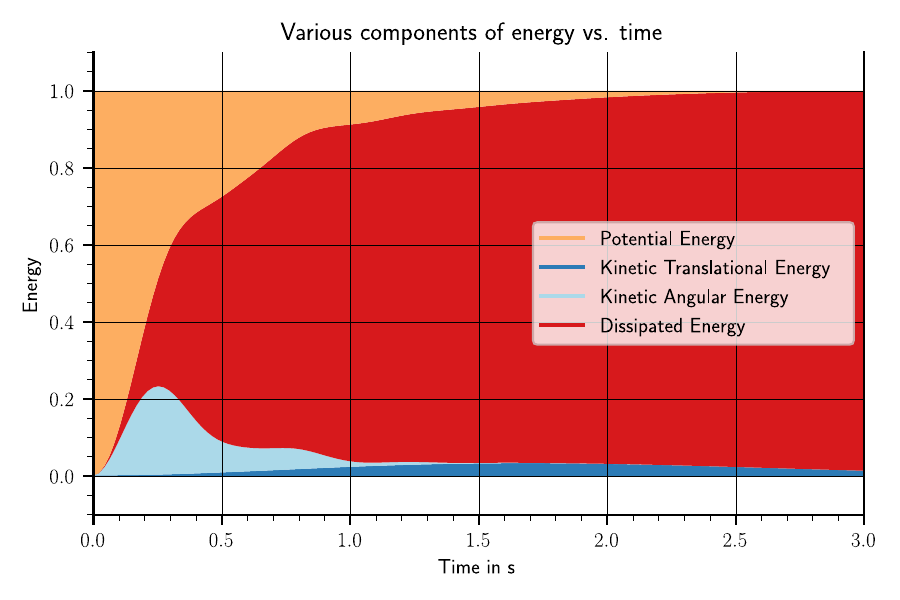}
    \subcaption{Average proportions of system energy over 3 seconds.}\label{sfig:quadratic_energy_components_average}
  \end{subfigure}
\end{tabular}
    \caption{Visualization of the performance of the quadratic controller characterized by $V_{Q,\theta}$ and $B_{Q,\theta}$, over 100 trajectories of rigid bodies with initial conditions sampled from $\mathbb{P}(\Gamma_0)$.}\label{fig:examples_quadratic}
\end{figure*}

\subsubsection{Nonlinear potential and damping injection:} \label{ssec:train:nonlinear_training}
Here we showcase the optimization of a nonlinear potential $V_{N,\theta}:SE(3)\rightarrow \mathbb{R}$ and a nonlinear damping injections $B_{N,\theta}(H,P) \in \mathbb{R}^{6\times 6}$. 
Both functions are parameterized by neural nets with one hidden layer of 64 neurons, using softplus and tanh activation functions. $V_{N,\theta}$ has 12 inputs (the components of $R$ and $p$, this is a projection of $H$) and 1 output, while $B_{N,\theta}$ has 18 inputs (components of $R$, $p$ and $P$) and 6 outputs, which are then put through an element-wise exponential function and turned into a diagonal 6 by 6 matrix to guarantee positive-definiteness of $B_{N,\theta}(H,P)\mathcal{I}$. 

The control law is of the form
\begin{equation}\label{eq:NN-control-law}
    W_{N,\theta}(H,P) = -\text{d}_{H}V_{N,\theta} - B_{N,\theta}(H,P)P\,.
\end{equation}

The parameters are optimized over $1000$ training epochs, using the ADAM optimizer with decay $\gamma = 0.999$, and a learning rate of $\eta = 0.001$. Additional parameters of the training are summarized in Appendix \ref{ssec:appB:Hyperparameters}, Table \ref{tab:appB:Hyperparameters_nonlinear_training}. 

The training progress is summarized in Figure \ref{fig:training_progress_NN-1}. 
It can be seen that the final loss of the nonlinear controller in Figure \ref{sfig:NN-1_loss} is equivalent to that of the quadratic controller Figure \ref{sfig:quadratic_loss}. In particular, the performance of a quadratic and a nonlinear controller for this scenario are close: the final angle and distance in Figures \ref{sfig:NN-1_th_final} and \ref{sfig:NN-1_d_final} are comparable to those of Figures \ref{sfig:quadratic_th_final} and \ref{sfig:quadratic_d_final}, respectively. 
The resulting controller's performance is shown in Figure \ref{fig:examples_NN-1}. Here the qualitative behavior shown in Figures \ref{sfig:NN-1_angle_3s}, \ref{sfig:NN-1_dist_3s} and \ref{sfig:NN-1_energy_components_average} resembles that of the quadratic case in Figures \ref{sfig:quadratic_angle_3s}, \ref{sfig:quadratic_dist_3s} and \ref{sfig:quadratic_energy_components_average}, respectively. 

\begin{figure*}[h]
\centering
\begin{tabular}{c c c}
  \begin{subfigure}{.33\textwidth}
    \scriptsize\includegraphics[width = \textwidth]{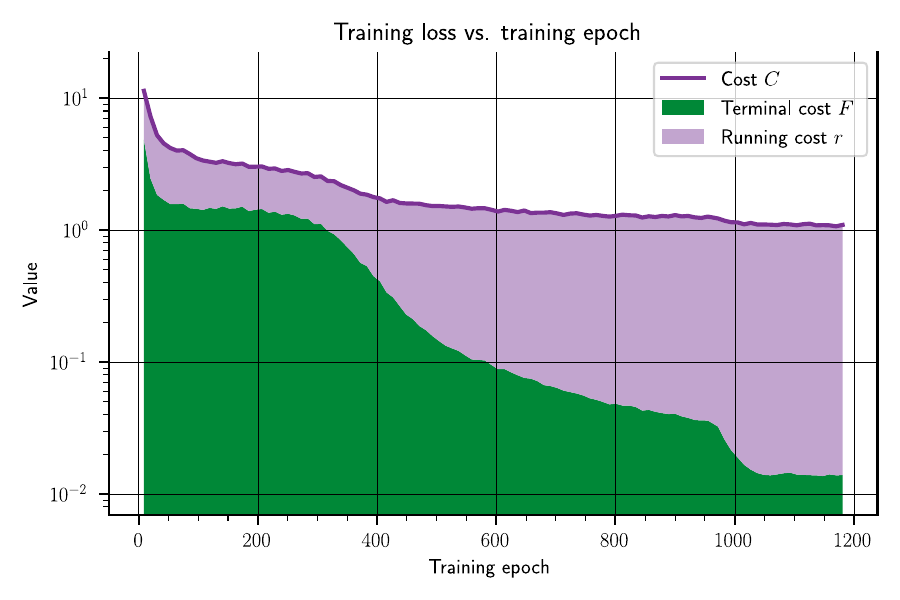}
    \subcaption{Loss, terminal loss and integral loss.}\label{sfig:NN-1_loss}
  \end{subfigure}
  &
  \begin{subfigure}{.33\textwidth}
    \scriptsize\includegraphics[width = \textwidth]{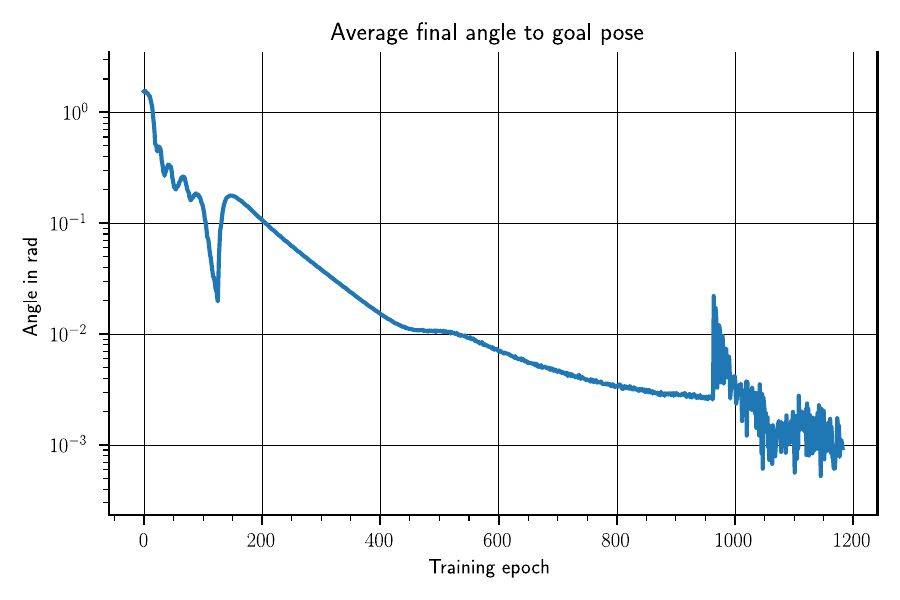}
    \subcaption{Average final angle towards goal pose.}\label{sfig:NN-1_th_final}
  \end{subfigure}
  &
  \begin{subfigure}{.33\textwidth}
    \scriptsize\includegraphics[width = \textwidth]{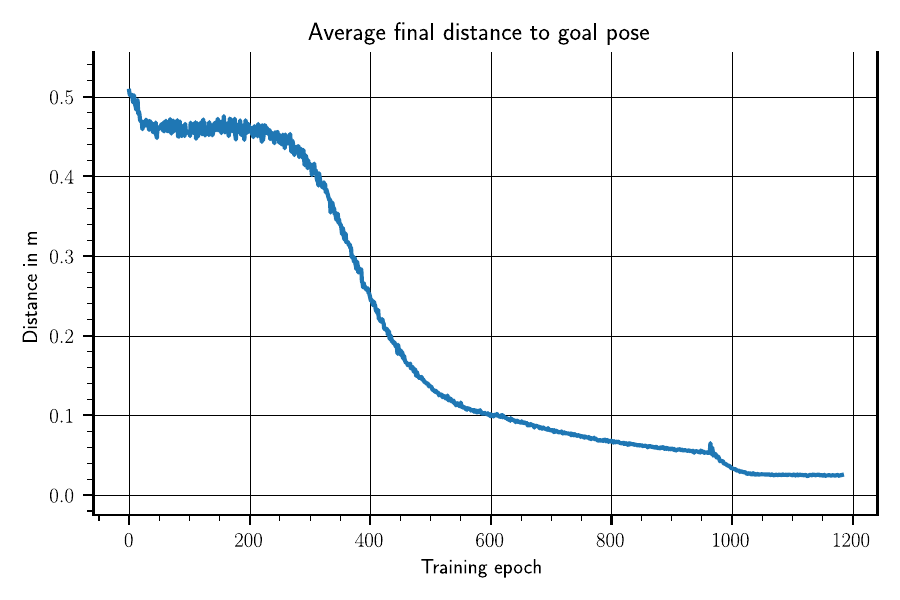}
    \subcaption{Average final distance towards goal pose.}\label{sfig:NN-1_d_final}
  \end{subfigure}
\end{tabular}
    \caption{Visualization of the training progress of the nonlinear controller characterized by $V_{N,\theta}$ and $B_{N,\theta}$. All figures show data averaged over 2048 sample trajectories at the given epoch, with initial conditions sampled from $\mathbb{P}$.}\label{fig:training_progress_NN-1}
\end{figure*}

\begin{figure*}[h]
\centering
\begin{tabular}{c c c}
  \begin{subfigure}{.33\textwidth}
    \scriptsize\includegraphics[width = \textwidth]{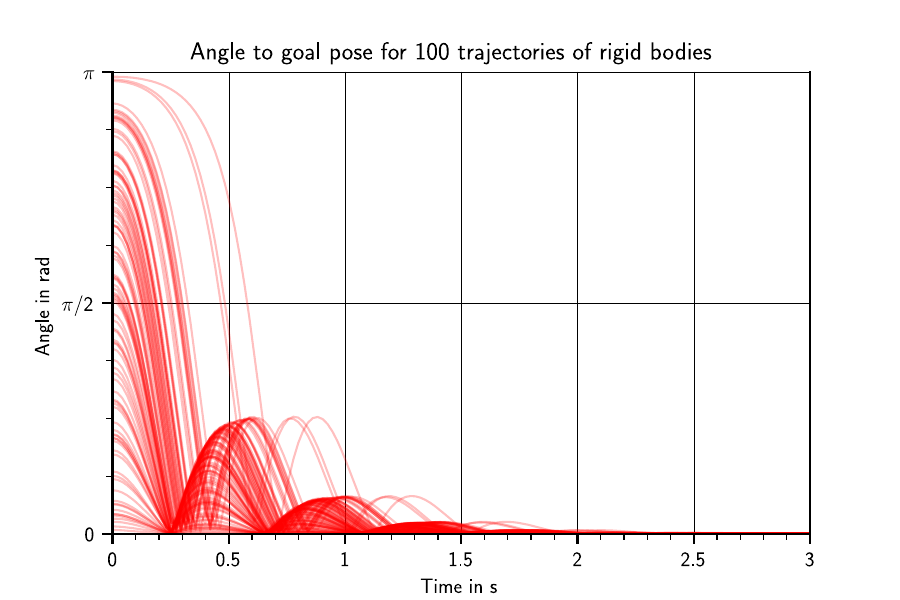}
    \subcaption{Angle towards goal pose over 3 seconds.}\label{sfig:NN-1_angle_3s}
  \end{subfigure}
  &
  \begin{subfigure}{.33\textwidth}
    \scriptsize\includegraphics[width = \textwidth]{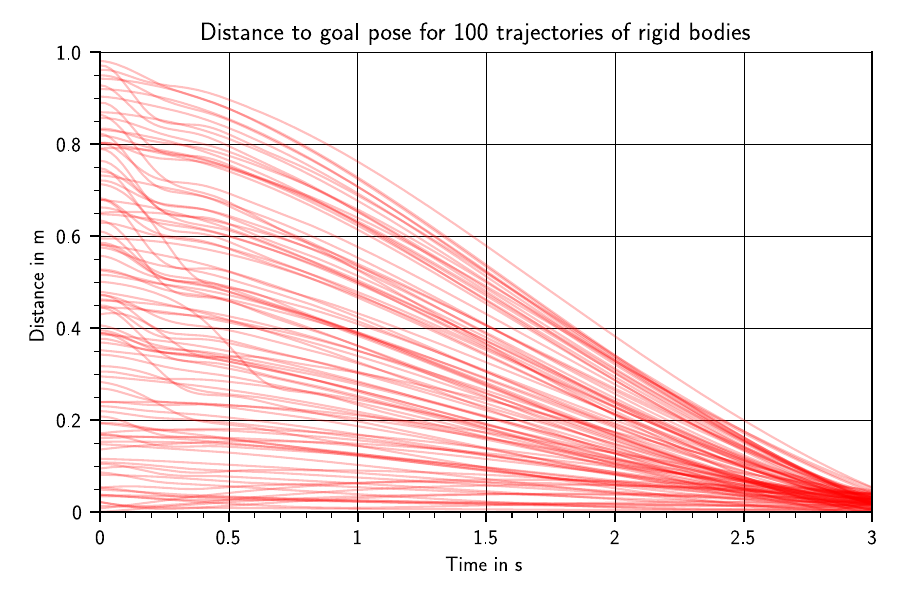}
    \subcaption{Distance towards goal pose over 3 seconds.}\label{sfig:NN-1_dist_3s}
  \end{subfigure}
  \begin{subfigure}{.33\textwidth}
    \scriptsize\includegraphics[width = \textwidth]{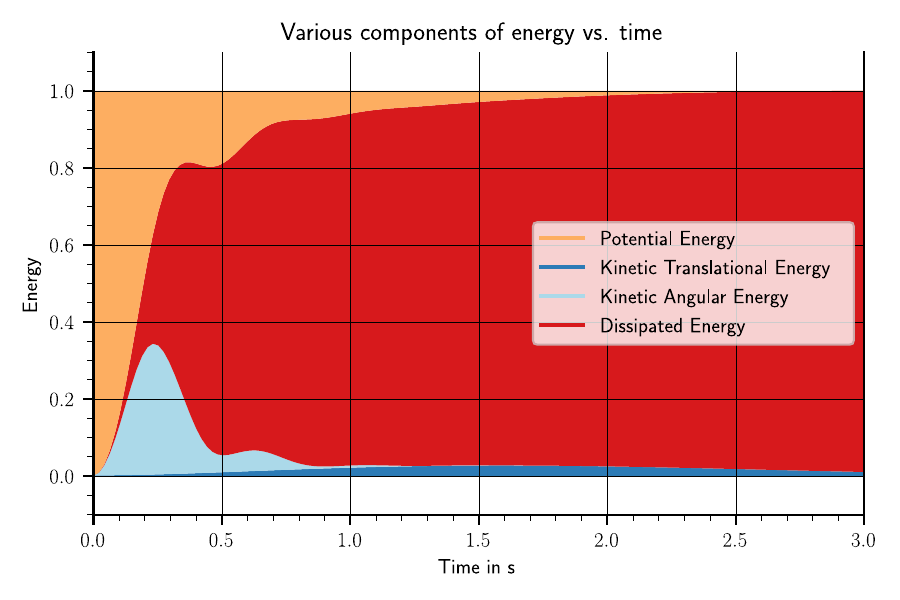}
    \subcaption{Average proportions of system energy over 3 seconds.}\label{sfig:NN-1_energy_components_average}
  \end{subfigure}
\end{tabular}
\caption{Visualization of the performance of the nonlinear controller characterized by $V_{N,\theta}$ and $B_{N,\theta}$, over 100 trajectories of rigid bodies with initial conditions sampled from $\mathbb{P}(\Gamma_0)$.}\label{fig:examples_NN-1}
\end{figure*}

\subsection{General potential shaping with gravity}\label{ssec:train:gravity}
We optimize an NN-parameterized potential and damping injection in a system with gravity in Section \ref{sssec:train:gravity:nonlinear_training_gravity}, and show the effect of an adapted target configuration in Section \ref{sssec:train:gravity:nonlinear_training_gravity_asymmetric}.

\subsubsection{Adapted running cost $r$:}
In the presence of a gravitational potential $V_g:SE(3)\rightarrow\mathbb{R}$ the momentum dynamics of a controlled rigid body \eqref{eq:dyn_SE3_dPdt} pick up an additional term $-\text{d}_H V_g$: 
\begin{align}
    \dot{P}_\theta(H,P) = \text{ad}_{\mathcal{I}^{-1}P}^\top P + W_\theta(H,P) - \text{d}_H V_g\,.
\end{align}
The gravitational potential is unbounded: it can therefore not be globally compensated for by a bounded potential $V_{N,\theta}$. To circumvent this issue, we separately implement gravity compensation $W_g(H) = \text{d}_{H}V_g$ by choosing the total control wrench as 
\begin{equation}
    W_\theta(H,P) = W_g(H) + W_{N,\theta}(H,P) \,.
\end{equation}
such that the momentum dynamics again read
\begin{align}
    \dot{P}_\theta(H,P) = \text{ad}_{\mathcal{I}^{-1}P}^\top P + W_{N,\theta}(H,P)\,.
\end{align}

To take gravity into account in the optimization, the only required adaptation is to use the adapted $W_\theta(H,P)$ in the running-cost \eqref{eq:specific_running_cost}. 

Minimizing the term $\|W_\theta(H,P)\|$ indirectly minimizes the required gravity compensation by reducing the total wrench exerted on the plant. Indeed, when $\|W_\theta\| = 0$ the dynamics are
\begin{equation}
    \dot{P}_\theta(H,P) = \text{ad}_{\mathcal{I}^{-1}P}^\top P - \text{d}_H V_g\,.
\end{equation}
Thus, for $\|W_\theta\| = 0$ the learned control-action $W_{N,\theta}(H,P)$ cancels the gravity compensation, such that the external gravitational potential is utilized to exert a force on the rigid body.

\subsubsection{Nonlinear potential and damping injection:}\label{sssec:train:gravity:nonlinear_training_gravity}
Here, we apply the ADAM optimizer with learning rate $\eta = 0.001$ and decay $\gamma = 0.999$ to optimize the nonlinear potential $V_{N,\theta}$ and damping $B_{N,\theta}$ for an adapted running cost. Additional parameters of the training are summarized in Appendix \ref{ssec:appB:Hyperparameters}, Table \ref{tab:appB:Hyperparameters_nonlinear_training_gravity}. The training progress is summarized in Figure \ref{fig:training_progress_NN-2-G}, and the resulting controller's performance is shown in Figure \ref{fig:examples_NN-2-G}. Notably, the results do not differ strongly from Figures \ref{fig:training_progress_NN-1} and \ref{fig:examples_quadratic}.

\begin{figure*}[h]
\centering
\begin{tabular}{c c c}
  \begin{subfigure}{.33\textwidth}
    \scriptsize\includegraphics[width = \textwidth]{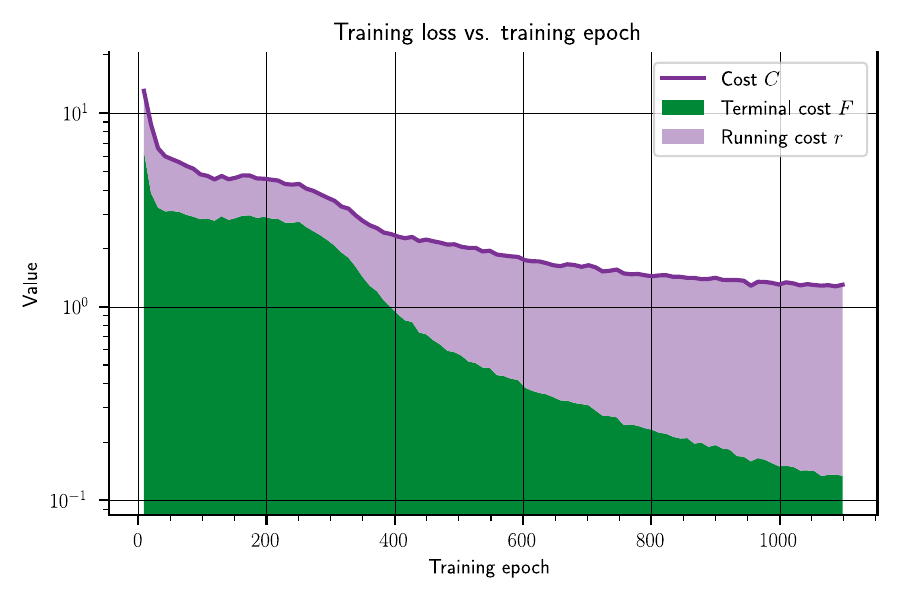}
    \subcaption{Loss, terminal loss and integral loss.}\label{sfig:NN-2-G_loss}
  \end{subfigure}
  &
  \begin{subfigure}{.33\textwidth}
    \scriptsize\includegraphics[width = \textwidth]{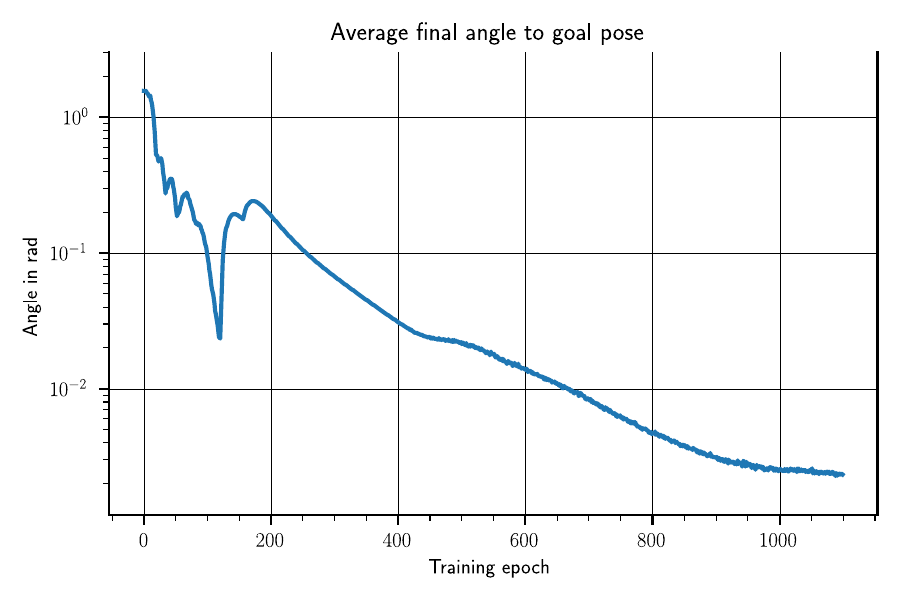}
    \subcaption{Average final angle towards goal pose.}\label{sfig:NN-2-G_th_final}
  \end{subfigure}
  &
  \begin{subfigure}{.33\textwidth}
    \scriptsize\includegraphics[width = \textwidth]{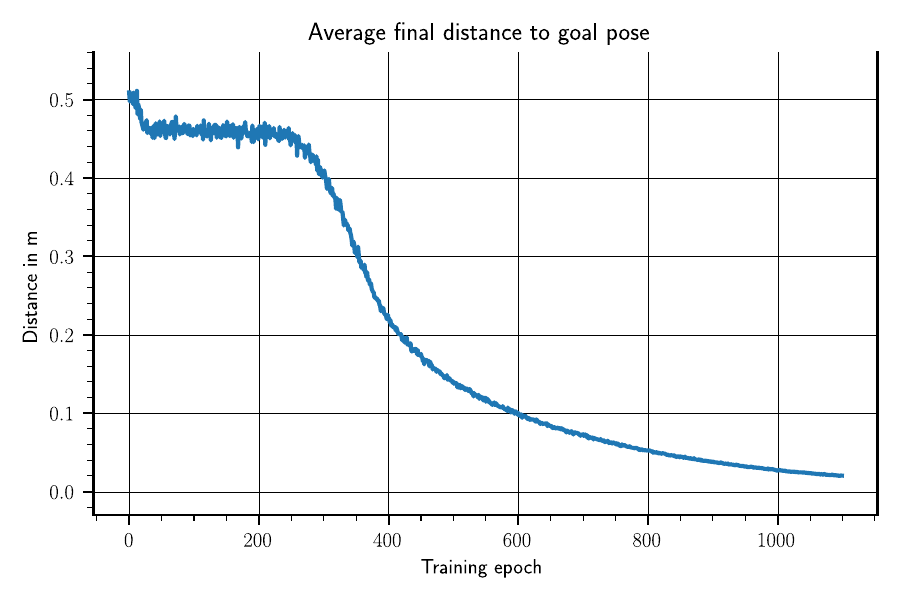}
    \subcaption{Average final distance towards goal pose.}\label{sfig:NN-2-G_d_final}
  \end{subfigure}
\end{tabular}
    \caption{Visualization of the training progress of the nonlinear controller characterized by $V_{N,\theta}$ and $B_{N,\theta}$, in the presence of gravity. All figures show data averaged over 2048 sample trajectories at the given epoch, with initial conditions sampled from $\mathbb{P}(\Gamma_0)$.}\label{fig:training_progress_NN-2-G}
\end{figure*}

\begin{figure*}[h]
\centering
\begin{tabular}{c c c}
  \begin{subfigure}{.33\textwidth}
    \scriptsize\includegraphics[width = \textwidth]{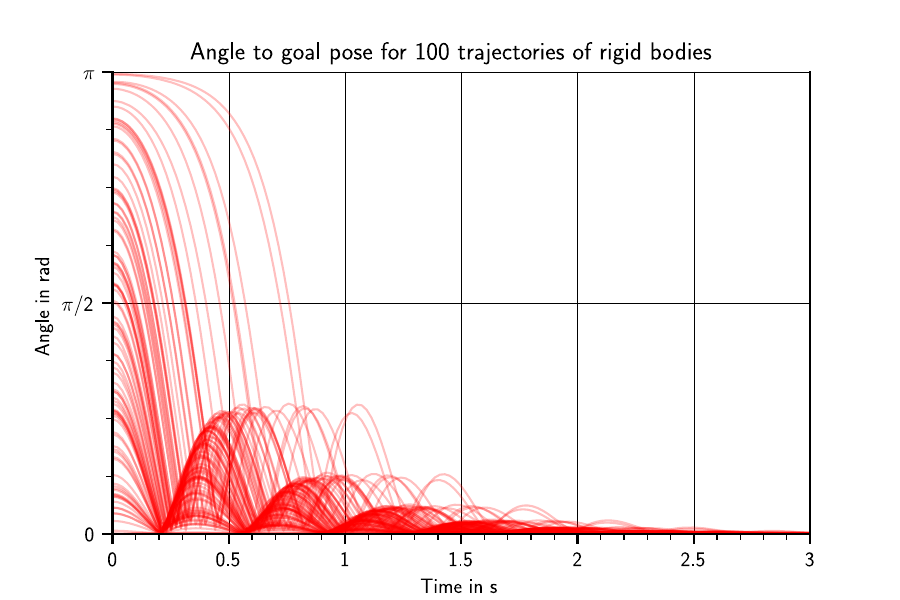}
    \subcaption{Angle towards goal pose over 3 seconds.}\label{sfig:NN-2_angle_3s}
  \end{subfigure}
  &
  \begin{subfigure}{.33\textwidth}
    \scriptsize\includegraphics[width = \textwidth]{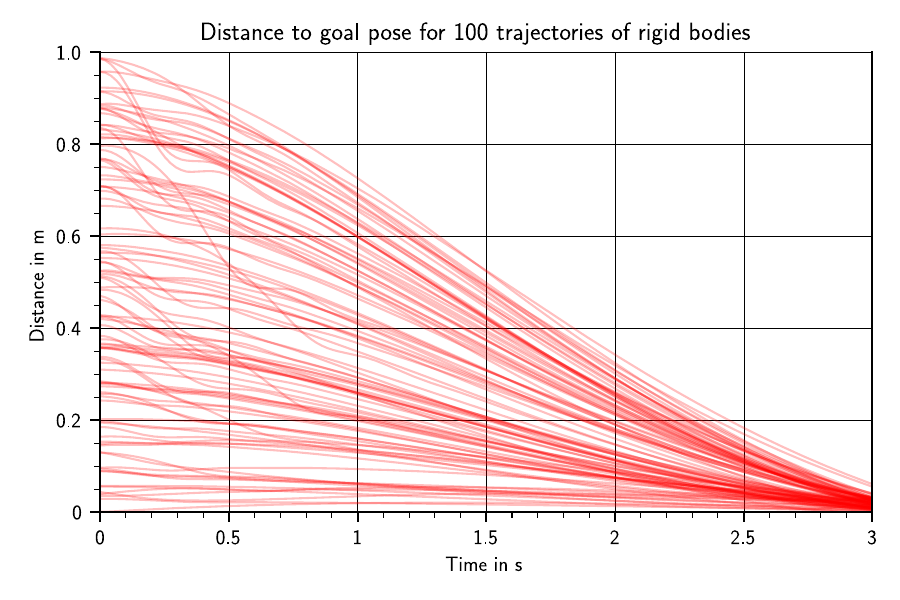}
    \subcaption{Distance towards goal pose over 3 seconds.}\label{sfig:NN-2_dist_3s}
  \end{subfigure}
  &
  \begin{subfigure}{.33\textwidth}
    \scriptsize\includegraphics[width = \textwidth]{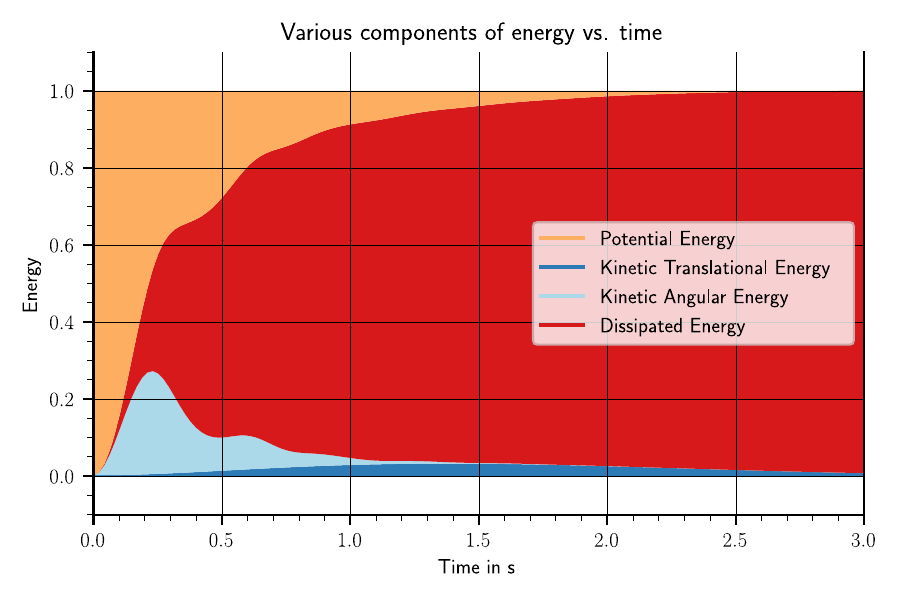}
    \subcaption{Average proportions of system energy over 3 seconds.}\label{sfig:NN-2_energy_components_average}
  \end{subfigure}
\end{tabular}
\caption{Visualization of the performance of the nonlinear controller characterized by $V_{N,\theta}$ and $B_{N,\theta}$, in the presence of gravity. The results show 100 trajectories of rigid bodies with initial conditions sampled from $\mathbb{P}(\Gamma_0)$.}\label{fig:examples_NN-2-G}
\end{figure*}

\subsubsection{Asymmetric initial distribution:}\label{sssec:train:gravity:nonlinear_training_gravity_asymmetric}
To better highlight the influence of the adapted running cost, and how the optimization in the presence of gravity differs from an optimization in the absence of gravity, an initial distribution asymmetric about the goal pose $H_F$ \eqref{eq: prior_dist} is introduced by choosing $$H_F =
\begin{bmatrix}
    I & p_F \\ 0 & 1
\end{bmatrix} $$ with $p_F = (0,0,-1)^T$. 

The parameters of this training coincide with those of the symmetric scenario, and are likewise summarized in Appendix \ref{ssec:appB:Hyperparameters}, Table \ref{tab:appB:Hyperparameters_nonlinear_training_gravity}.
The training progress is summarized in Figure \ref{fig:training_progress_NN-3-AG}, and the resulting controller's performance is shown in Figure \ref{fig:examples_NN-3-AG}.

\begin{figure*}[h]
\centering
\begin{tabular}{c c c}
  \begin{subfigure}{.33\textwidth}
    \scriptsize\includegraphics[width = \textwidth]{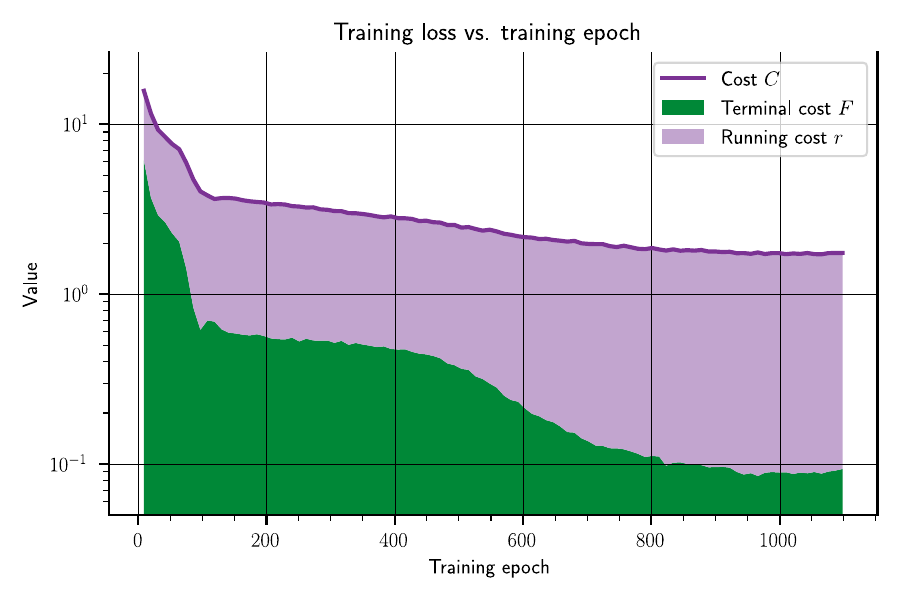}
    \subcaption{Loss, terminal loss and integral loss.}\label{sfig:NN-3-AG_loss}
  \end{subfigure}
  &
  \begin{subfigure}{.33\textwidth}
    \scriptsize\includegraphics[width = \textwidth]{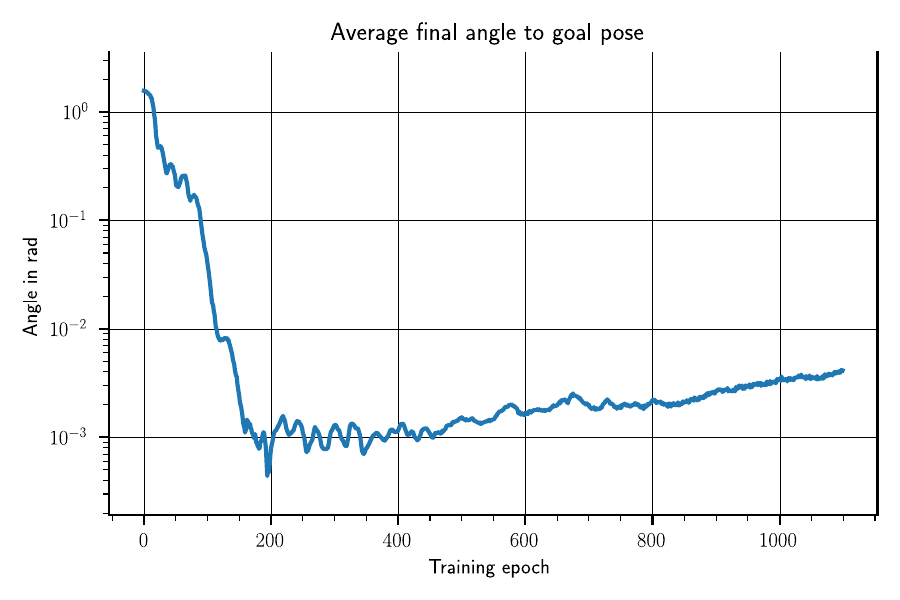}
    \subcaption{Average final angle towards goal pose.}\label{sfig:NN-3-AG_th_final}
  \end{subfigure}
  &
  \begin{subfigure}{.33\textwidth}
    \scriptsize\includegraphics[width = \textwidth]{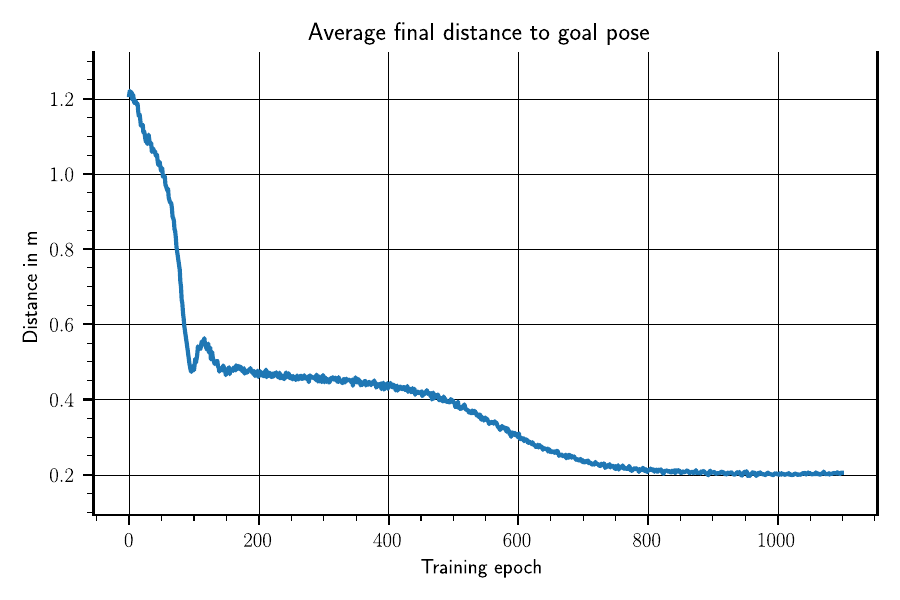}
    \subcaption{Average final distance towards goal pose.}\label{sfig:NN-3-AG_d_final}
  \end{subfigure}
\end{tabular}
    \caption{Visualization of the training progress of the nonlinear controller characterized by $V_{N,\theta}$ and $B_{N,\theta}$, in the presence of gravity and with an initial distribution whose mean is above the target position. All figures show data averaged over 2048 sample trajectories at the given epoch, with initial conditions sampled from $\mathbb{P}(\Gamma_0)$.}\label{fig:training_progress_NN-3-AG}
\end{figure*}

\begin{figure*}[h]
\centering
\begin{tabular}{c c c}
  \begin{subfigure}{.33\textwidth}
    \scriptsize\includegraphics[width = \textwidth]{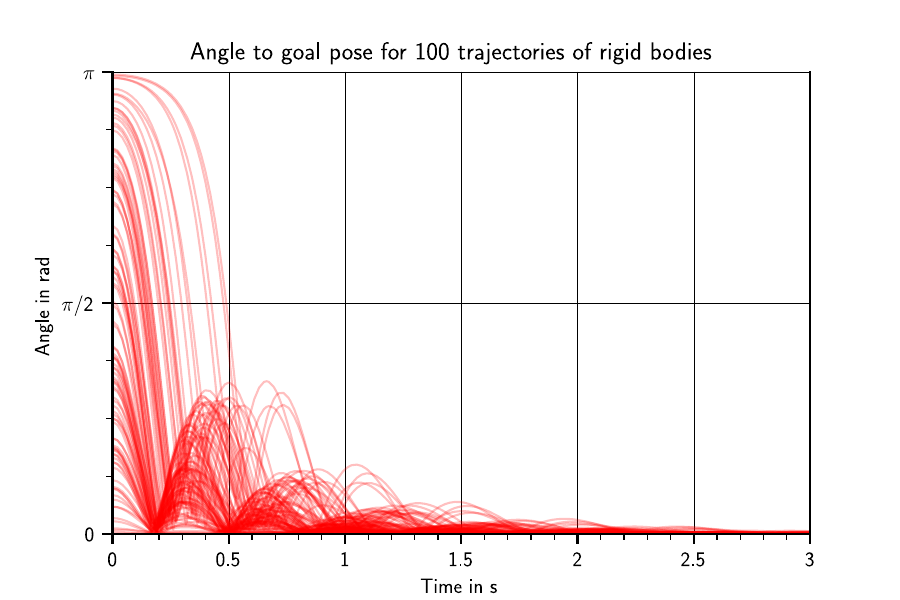}
    \subcaption{Angle towards goal pose over 3 seconds.}\label{sfig:NN-3_angle_3s}
  \end{subfigure}
  &
  \begin{subfigure}{.33\textwidth}
    \scriptsize\includegraphics[width = \textwidth]{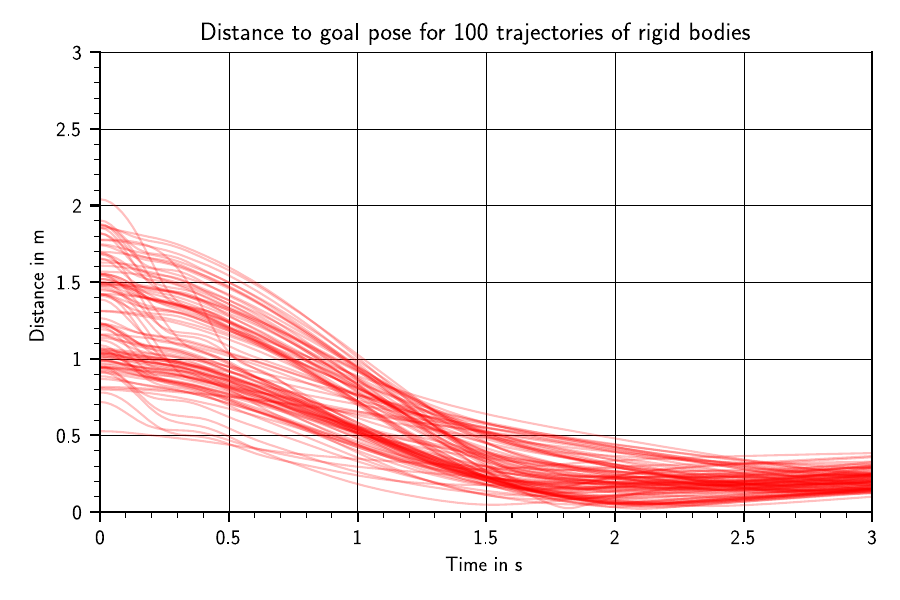}
    \subcaption{Distance towards goal pose over 3 seconds.}\label{sfig:NN-3_dist_3s}
  \end{subfigure}
  &
  \begin{subfigure}{.33\textwidth}
    \scriptsize\includegraphics[width = \textwidth]{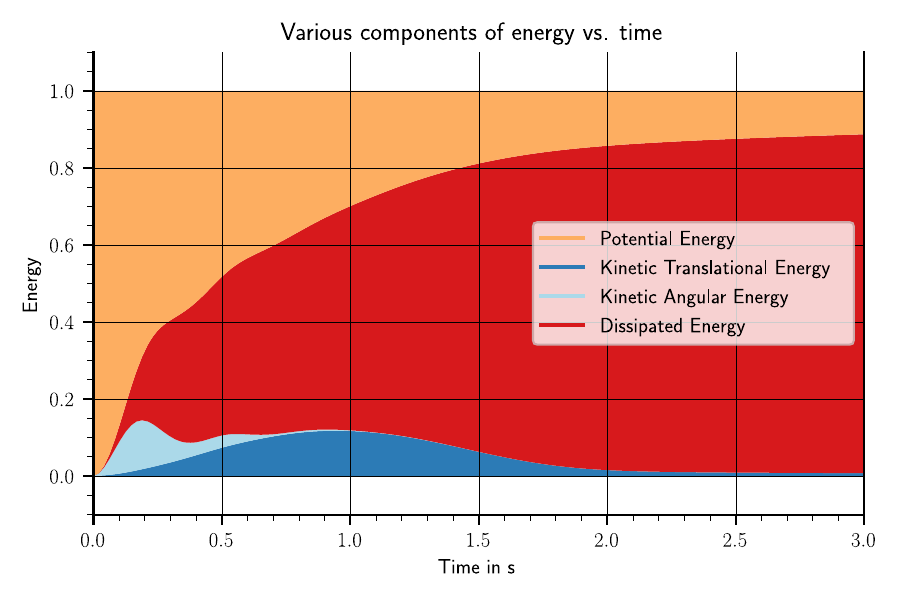}
    \subcaption{Average proportions of system energy over 3 seconds.}\label{sfig:NN-3_energy_components_average}
  \end{subfigure}
\end{tabular}
\caption{Visualization of the performance of the nonlinear controller characterized by $V_{N,\theta}$ and $B_{N,\theta}$, in the presence of gravity and with an initial distribution whose mean is above the target position. The results show 100 trajectories of rigid bodies with initial conditions sampled from $\mathbb{P}(\Gamma_0)$.}\label{fig:examples_NN-3-AG}
\end{figure*}

\section{Discussion}\label{sec:discussion}

\subsection{Neural ODEs on Lie groups} 

The proposed formulation of neural ODEs on Lie groups immediately applies to arbitrary matrix Lie groups, where parameterized maps can be learned with a global validity. 
The optimization of Neural ODEs on Lie groups by the gradient descent via the generalized adjoint method is a scalable approach. The key aspects that contribute to this scalability are:
First, the generalized adjoint method on Lie groups preserves the memory efficiency of the generalized adjoint method used for neural ODEs on $\mathbb{R}^k$. 
Second, the formulation of the adjoint dynamics at the algebra level achieves a dimensionality reduction with respect to extrinsic formulations of \cite{Duong2021_nODE_SE3,Falorsi2020}. Finally, this formulation at the algebra level also alleviates the need for chart-switches of the adjoint state, and it allows for the use of the compact expression \eqref{eq:grad_G} of the gradient, bypassing the need for gradient computations in local charts. 

The work can be generalized further: Theorem \ref{thm:adjoint_method_matrix_lie_group} assumes the cost to be of the form \eqref{eq:single_trajectory_cost}, while the derivation in Appendix \ref{app:nodes_on_manifolds} in principle allows for a more general choice of cost that may be of interest in e.g. learning of periodic trajectories \citep{Wotte2023}.
The accompanying code is currently written specifically for the Lie group $SE(3)\times \mathbb{R}^6$, and future work will produce code that is applicable to other matrix Lie groups as well.

\subsection{Optimal Potential Shaping}
The optimization of an NN-parameterized potential and damping injection was successful and the large number of parameters used in the optimization confirms that it scales to the large parameter scenario. The optimization was also successful when including gravity in a nonlinear running cost. Stability was guaranteed by design, by implementing the requirements of Theorem \ref{thm:stability} on the level of architecture and activation functions. As a further advantage the resulting controller is global on $SE(3)$, as opposed to only being applicable in a limited chart-region. 

Regarding limitations of the approach, the numerical stability of the adjoint method on $SE(3)$ was observed to strongly depend on the smoothness of the running cost, which suggests added value in considering different Lie group integrators that accommodate this lack of smoothness. 
Lastly, while the structure of the presented controller is highly interpretable and the various components of the energy are readily visualized, the space of possible initial conditions and trajectories remains large, and the high dimensional state-space obscures low-level properties and a deep understanding of the eventual controller, beyond safety guarantees and numerical verification of stability. 

Alternative choices for the final and running costs, as well as the weights in these costs are worth investigating. The design space of possible controllers is also large and other control architectures may be advantageous.
In future work the controller will be applied to a real drone, and other cost functions and control structures will be investigated.

\section{Conclusion}\label{sec:conclusion}


Lie groups are ubiquitous in engineering, and so are dynamic systems on Lie groups. We proposed a method for dynamics optimization that works on arbitrary, finite dimensional Lie groups and for a large class of cost functions. The resulting method is highly scalable, and more compact than alternative manifold formulations.
The key steps in the formulation related to using canonical Lie group structure to create a compact gradient descent algorithm: we phrased the generalized adjoint method at the Lie algebra level, we utilize a compact expression for the gradient as an element of the dual to the Lie algebra, and we use a generic Lie group integrator for dynamics integration. 
The method was successfully applied to optimize a controller for a rigid body that is globally valid on the Lie group SE(3). A key aspect of choosing the class of controllers was stability by design, which guided the architecture of the neural nets that parameterize the potential energy shaping and damping injection controller.











\begin{dci} 
    The authors declare that they have no conflicts of interest.
\end{dci}

\begin{funding} 
    This research was supported by the PortWings project funded by the European Research Council [Grant Agreement No. 787675].
\end{funding}

\theendnotes

\begin{sm}
   \appendix 
   \section{The generalized adjoint method on Lie groups}\label{app:A}

In this Appendix the generalized adjoint method on matrix Lie groups (Theorem \ref{thm:adjoint_method_matrix_lie_group}) is derived in four steps. 


\subsection{Notation}
Given functions $\Phi_1 \in C^k(\mathcal{M},\mathcal{N})$, $\Phi_2 \in C^k(\mathcal{N},\mathcal{P})$, their composition is denoted $\Phi_2\circ \Phi_1 \in C^k(\mathcal{M},\mathcal{P})$.
$\Omega^k(\mathcal{M})$ denotes the set of $k$-forms over $\mathcal{M}$. 
The exterior derivative is denoted as $\text{d}:\Omega^k(\mathcal{M})\rightarrow \Omega^{k+1}(\mathcal{M})$, and the wedge product as $\wedge:\Omega^{q}(\mathcal{M})\times \Omega^{p}(\mathcal{M})\rightarrow \Omega^{q+p}(\mathcal{M})$. 
Given a product manifold $\mathcal{M}\times\mathcal{N}$, a function $V\in C^1(\mathcal{M}\times\mathcal{N},\mathbb{R})$, with $x\in \mathcal{M}$, $y\in\mathcal{N}$, denote by $\big(\text{d}_x V\big)(y)\in T^*_x\mathcal{M}$ the partial gradient at $x\in\mathcal{M}$.
Given a vector field $f\in \Gamma(T\mathcal{M})$ and a $k$-form $\omega \in \Omega^k(\mathcal{M})$ denote by $\i_f:\Omega^k(\mathcal{M}) \rightarrow \Omega^{k-1}(\mathcal{M})$ the insertion operator $\i_f(\omega) := \omega(f)$. Further denote by $\mathcal{L}_f:\Omega^k(\mathcal{M})\rightarrow\Omega^k(\mathcal{M})$ the Lie derivative with respect to $f$.

\subsection{Hamiltonian systems on Lie groups}\label{app:background_Hamiltonian_systems_on_Lie_groups}
We briefly review Hamiltonian systems on manifolds, on Lie groups and on matrix Lie groups. For a detailed introduction to Hamiltonian systems see e.g., \cite{Marsden1999}. 

Define a symplectic manifold as $(\mathcal{M},\omega)$, with $\mathcal{M}$ a manifold and $\omega \in \Omega^2(\mathcal{M})$ the symplectic form. For Hamiltonian dynamics on manifolds, we are interested in the case where $\mathcal{M} = T^*\mathcal{Q}$ is the cotangent bundle of some manifold $\mathcal{Q}$. Specializing to Lie groups, we investigate the case where $\mathcal{Q} = G$ for some Lie group $G$.


Given coordinate maps $q^i:\mathcal{Q}\rightarrow\mathbb{R}$ on $\mathcal{Q}$ and induced coordinates $p_i$ in the basis $\text{d}q^i$ on $T_q^*\mathcal{Q}$, the symplectic form is canonically defined as $\omega = \text{d}\theta = \text{d}p_i \wedge \text{d}q^i$, with $\theta = p_i \text{d}q^i$ the coordinate independent tautological one-form. 
Then a Hamiltonian $H \in C^1(T^*\mathcal{Q},\mathbb{R})$ implicitly defines a unique vector field $X_H \in \Gamma(T^*\mathcal{Q})$ by demanding that
\begin{equation}
    \text{d}H(Y) = \omega(X_H, Y) 
\end{equation}
holds for any $Y \in \Gamma(T^*\mathcal{Q})$. In coordinates induced by $q^i, p_i$, the corresponding vector field $X_H$ has the components 
\begin{align}
    \dot{q}^i &= \frac{\partial H}{\partial p_i} \,, \label{eq:Ham_q} \\
    \dot{p}_i &= -\frac{\partial H}{\partial q^i} \label{eq:Ham_p} \,.
\end{align}

On a Lie group $G$ a different formulation is possible: the group structure allows the identification $T^*G \equiv G \times \mathfrak{g}^*$, where $\mathfrak{g}^* = T^*_e G$ is the dual of the Lie algebra $\mathfrak{g}$. 
In the left identification, the left-translation map $L_g:G\rightarrow G$ associates any cotangent space $T^*_g G$ with $\mathfrak{g}^*$ by the pullback ${L_g}^*:T_g^*G \rightarrow \mathfrak{g}^*$. 
Using this identification, the Hamiltonian $\mathcal{H}:G\times \mathfrak{g}^* \rightarrow \mathbb{R}$ is defined in terms of $H:T^*G\rightarrow \mathfrak{g}$ as
\begin{equation}\label{eq:HamLie_to_Ham}
    \mathcal{H}(g,\bar{P}) = H\big(g, (L_{g^{-1}}^*\bar{P})\big)\,.
\end{equation}
Then the Hamiltonian vector-field as a section in $\Gamma(TG\times T\mathfrak{g}^*)$ is:
\begin{align} 
    \dot{g} &= {L_{g}}_*(\text{d}_{\bar{P}} \mathcal{H})\,, \label{eq:Lie_Ham} \\
    \dot{\bar{P}} &= - \bar{\text{d}}_{g} \mathcal{H} + \text{ad}_{\text{d}_{\bar{P}} \mathcal{H}}^* \bar{P} \label{eq:Lie_Ham_Co} \,,
\end{align}
where $\text{d}_{\bar{P}} \mathcal{H} \in T^*_{\bar{P}} \mathfrak{g}^* \cong \mathfrak{g}^*$, $\bar{\text{d}}_{g} \mathcal{H} \in \mathfrak{g}^*$ as in Equation \eqref{eq:bar_grad_G}, $\text{ad}^*:\mathfrak{g}\times\mathfrak{g}^*\rightarrow\mathfrak{g}^*$ is the dual of the adjoint map $\text{ad}$, defined by $(\text{ad}^*_{\tilde{A}}\bar{P})(\tilde{B}) = \bar{P}(\text{ad}_{\tilde{A}}\tilde{B})$. 

For a matrix Lie group the equations \eqref{eq:Lie_Ham} and \eqref{eq:Lie_Ham_Co} read, in terms of component matrices:
\begin{align} 
    \dot{g} &= g \Lambda (\frac{\partial \mathcal{H}}{\partial P} )\,, \label{eq:Lie_Ham_Matrix} \\
    \dot{P} &= - \text{d}_g \mathcal{H} + \text{ad}_{\frac{\partial\mathcal{H}}{\partial P} }^\top P \label{eq:Lie_Ham_Co_Matrix} \,.
\end{align}
Here $P\in\mathbb{R}^n$ and $\frac{\partial}{\partial P}$ is the usual partial derivative, $\text{d}_g \mathcal{H} \in \mathbb{R}^n$ is interpreted as in \eqref{eq:grad_G} and $\Lambda:\mathbb{R}^n\rightarrow\mathfrak{g}$ as in \eqref{eq:abstract_tilde}.


\subsection{The adjoint sensitivity}\label{ssec:app_adjoint_sensitivity}

Given are a manifold $\mathcal{M}$, a Lipschitz vector field $f\in \Gamma(T\mathcal{M})$, the associated flow $\Psi_{f}^t:\mathcal{M}\rightarrow \mathcal{M}$ and a differentiable scalar-valued function $C:\mathcal{M}\rightarrow \mathbb{R}$. 
A solution $x(t) \in \mathcal{M}$ of the dynamics $f$ is given by
\begin{equation}
    x(t) := \Psi_f^t(x_0) \,,\; x_0 \in \mathcal{M}\,.
\end{equation}
We are interested in computing the gradient
\begin{equation}\label{eq:grad_over_flow_1}
    \text{d} (C\circ \Psi_f^T)(x_0) = (\Psi_f^T)^* \text{d}C \big(x(T)\big) \,,
\end{equation}
whose expression is given by Theorem \ref{thm:adjoint_sensitivity_manifold} \citep{Falorsi2020,BulloLewis2005_Supplementary}.
To the best of the authors' knowledge, the presented derivation is a contribution with respect to the existing literature and is an alternative to the one presented in \cite{BulloLewis2005_Supplementary}.

\begin{theorem}[Adjoint sensitivity on manifolds]\label{thm:adjoint_sensitivity_manifold}
    The gradient of a function $C\circ\Psi_f^T$ is 
    \begin{equation}\label{eq:grad_over_flow_2}
        \text{d} (C\circ \Psi_f^T) = \lambda(0)\,,
    \end{equation}
    where $\lambda(t) \in T^*_{x(t)}\mathcal{M}$ is the adjoint state. In a local chart $(U,X)$ of $\mathcal{M}$ with induced coordinates on $TU$ and $T^*U$, $x(t)$ and $\lambda(t)$ satisfy the dynamics
    \begin{align}
        \dot{x}^j &= f^j(x)\,,\; x(0) = x_0 \,, \label{eq:adjoint_sensitivity_manifold_x}\\
        \dot{\lambda}_i &= - \lambda_j \frac{\partial}{\partial x^i} f^j(x) \,, \; \lambda(T) = \text{d}C\big(x(T)\big) \label{eq:adjoint_sensitivity_manifold}\,.
    \end{align}
\end{theorem}
\begin{proof}

Define the adjoint state $\lambda(t) \in T^*_{x(t)}\mathcal{M}$ as
\begin{align}
    \lambda(t) &:= (\Psi_f^{T-t})^* \lambda_T \,,\; \lambda_T\in T_{x(T)}^*\mathcal{M} \,. \label{eq:def_adjoint_state_mfd}
\end{align}
Let $\lambda_T = \text{d}C\big(x(T)\big)$, then Equation \eqref{eq:grad_over_flow_2} is recovered:
\begin{equation}
    \lambda(0) = (\Psi_f^{T})^*\lambda_T = (\Psi_f^{T})^* \text{d}C \big(x(T)\big) = (\text{d}C\circ\Psi^T_f)(x_0)\,,
\end{equation}
where the final step uses Equation \eqref{eq:grad_over_flow_1}. 

A derivation of the dynamics governing $\lambda(t)$ constitutes the remainder of this proof.
To this end, note that the Lie derivative of $\lambda(t)$ is
\begin{align}
    \mathcal{L}_f \lambda(t) 
    &= \frac{d}{ds}\big((\Psi^s_f)^* \lambda(t+s)\big)_{s=0} \\ 
    &= \frac{d}{ds}\big((\Psi^s_f)^* (\Psi_f^{T-t-s})^* \lambda_T\big)_{s=0} \nonumber \\
    &= \frac{d}{ds}\big((\Psi_f^{T-t})^* \lambda_T\big)_{s=0} = 0\nonumber \,.
\end{align}
Instead of a curve $\lambda(t)$, consider a 1-form $\lambda\in \Omega^1(\mathcal{M})$ (denoted as $\lambda$ by an abuse of notation) that satisfies $\mathcal{L}_f \lambda = 0$. This allows to apply Cartan's formula, which we express in a local chart:
\begin{align}\label{app:eq:cartan}
    \mathcal{L}_f \lambda = \; & \text{d}(\i_f \lambda) + \i_f \text{d}\lambda  \\
    = & \frac{\partial}{\partial x^i} (\lambda_j f^j) \text{d}x^i + (\frac{\partial}{\partial x^k} \lambda_i) (\text{d}x^k \wedge \text{d} x^i)(f^j \frac{\partial}{\partial x^j}) \nonumber \\
    =\,  &\frac{\partial}{\partial x^i} (\lambda_j f^j) \text{d}x^i + (\frac{\partial}{\partial x^k} \lambda_i) f^k \text{d} x^i - (\frac{\partial}{\partial x^k} \lambda_i) f^i \text{d}x^k \nonumber \\
    =\, & \lambda_j \frac{\partial}{\partial x^i}  f^j \text{d}x^i + (\frac{\partial}{\partial x^k} \lambda_i) f^k \text{d} x^i = 0 \,. \nonumber
\end{align}
In terms of components, one obtains the partial differential equation
\begin{equation} \label{app:eq:cartan-components}
    \lambda_j \frac{\partial}{\partial x^i}  f^j + f^k \frac{\partial}{\partial x^k} \lambda_i = 0 \,.
\end{equation}
Impose that $\lambda(t) = \lambda\big(\Psi^t_f(x_0)\big)$, then 
\begin{equation}\label{app:eq:dlambda_dt}
    \dot{\lambda}_i = f^k\frac{\partial}{\partial x^k} \lambda_i\,.
\end{equation}
Combining Equations \eqref{app:eq:cartan-components} and \eqref{app:eq:dlambda_dt} leads to Equation \eqref{eq:adjoint_sensitivity_manifold}:
\begin{equation}
    \dot{\lambda}_i = - \lambda_j \frac{\partial}{\partial x^i}  f^j \,.
\end{equation} 
\qed
\end{proof}

\begin{remark}
    Theorem \ref{thm:adjoint_sensitivity_manifold} also holds for time-dependent dynamics $f(x,t)$. This is because time-dependent dynamics $f(x,t)$ on $\mathcal{M}$ can be recast as time-indepent dynamics on a space-time product manifold $\mathcal{N} = \mathcal{M}\times\mathbb{R}$, where we identify $ y = (x,t) \in \mathcal{N}$ with $x\in\mathcal{M}$ and $t \in \mathbb{R}$ and define dynamics $f^{\mathcal{N}}(y) = f(x,t) + \frac{\partial}{\partial t}$. The caveat is that $t$ then needs to be treated as a coordinate on $\mathcal{N}$, and should not be confused with the elapsed time $s$ in the flow $\Psi^s_{f^\mathcal{N}}:\mathcal{N}\rightarrow\mathcal{N}$.
\end{remark}

Theorem \ref{thm:adjoint_sensitivity_manifold} can be recast into a Hamiltonian form:

\begin{lemma}[Hamiltonian form of adjoint sensitivity]\label{lemma:hamiltonian_adjoint_sensitivity_manifold}
Define the so-called control Hamiltonian
\begin{equation}
    H_c:T^*\mathcal{M} \rightarrow \mathbb{R}\,;\; H_c(x,\lambda) = \lambda\big(f(x)\big)\,.
\end{equation}
Then Hamilton's equations \eqref{eq:Ham_q} and \eqref{eq:Ham_p} agree with the equations for the adjoint sensitivity \eqref{eq:adjoint_sensitivity_manifold_x} and \eqref{eq:adjoint_sensitivity_manifold}, respectively:
\begin{align}
        \dot{x}^j &=\frac{\partial H_c}{\partial \lambda_j} = f^j\,, \label{eq:hamiltonian_adjoint_sensitivity_manifold_q} \\
        \dot{\lambda}_{i} &= - \frac{\partial H_c}{\partial x^i}  
        = - \lambda_{j} \frac{\partial}{\partial x^i} f^j  \label{eq:hamiltonian_adjoint_sensitivity_manifold_p}\,. 
    \end{align}
\end{lemma}

\begin{remark}
    The vector field $f^\uparrow \in \Gamma(TT^*\mathcal{M})$ defined by equations \eqref{eq:hamiltonian_adjoint_sensitivity_manifold_q} and \eqref{eq:hamiltonian_adjoint_sensitivity_manifold_p} is called the co-tangent lift of $f$.
\end{remark}

\subsection{Neural ODEs on Manifolds}\label{app:nodes_on_manifolds}


%



We phrase the optimal control problem on a manifold $\mathcal{M}$. Let $\theta \in \mathcal{P}$ denote a parameter, and define the parameterized, dynamic system
\begin{equation}\label{app:eq:parameterized_f_mfd}
    f:\mathcal{P}\rightarrow\Gamma(T\mathcal{M}) \,; \; f_\theta := f(\theta) \in \Gamma(T\mathcal{M}) \,.
\end{equation} 
We allow for time-dependent dynamics, and consider instead dynamics on the space-time manifold $\mathcal{N} = \mathcal{M}\times\mathbb{R}$, with dynamics $f^\mathcal{N}_\theta = f_\theta + \frac{\partial}{\partial t} \in \Gamma(T\mathcal{N})$. To avoid confusion with $t$, we denote the elapsed time by $s$, i.e., $t(0) = t_0$ and $t(s) = t_0 + s$.

With the flow $\Psi_{f_\theta}^s:\mathcal{N}\rightarrow\mathcal{N}$, define the cost $C^T_f:\mathcal{N}\rightarrow\mathbb{R}$ as
\begin{align}\label{eq:single_trajectory_cost_mfd}
    C^T_{f_\theta}\big((x_0,t_0),\theta\big) =& F\big(\Psi^T_{f_\theta}(x_0,t_0),\theta\big)\\ &+ \int_0^T r\big(\Psi^s_{f_\theta}(x_0,t_0),\theta,s\big) \text{d}s \,. \nonumber
\end{align}
The parameter gradient $\text{d}_\theta C^T_{f_\theta}\big((x_0,t_0),\theta\big) \in T^*_\theta \mathcal{P}$ is then computed by Theorem \ref{thm:adjoint_method_mfd}:

\begin{theorem}[Generalized Adjoint Method on Manifolds]\label{thm:adjoint_method_mfd}
    Given the dynamics \eqref{app:eq:parameterized_f_mfd} and the cost \eqref{eq:single_trajectory_cost_mfd}, the parameter gradient $\text{d}_\theta C^T_{f_\theta}\big((x_0,t_0),\theta\big) \in T^*_\theta \mathcal{P}$ is computed by 
    \begin{equation} \label{eq:grad_cost_mfd}
       \text{d}_\theta C^T_{f_\theta}\big((x_0,t_0),\theta\big) 
       = 
       \text{d}_\theta F + \int_0^T\text{d}_\theta \big(\lambda(f_\theta) +  r \big) \text{d}s\,.
    \end{equation}
    where the state $x(s) \in \mathcal{M}$ and adjoint state $\lambda(s) \in T^*_{x(s)}\mathcal{M}$ satisfy
    \begin{align}
        \dot{x}^j &= f^j\,,\; x(0) = x_0 \,,\; t(0) = t_0 \,, \label{eq:adjoint_method_manifold_x}\\
        \dot{\lambda}_{i} &= - \lambda_{j} \frac{\partial}{\partial x^i} f^j  - \frac{\partial r}{\partial x^i}\,,\; \lambda(T) = \text{d}C\big(x(T)\big) \label{eq:adjoint_method_manifold}\,.
    \end{align}
\end{theorem}





\begin{proof}

Define the augmented state space as $\mathcal{M}' = \mathcal{M} \times \mathcal{P} \times \mathbb{R} \times \mathbb{R}$ with state $x' :=(x, \theta, L, t) \in \mathcal{M}'$. In addition, define the augmented dynamics $f_{\text{aug}} \in \Gamma(T\mathcal{M}')$ as
\begin{equation}
    f_{\text{aug}}(x') = 
    \begin{pmatrix}
        f_\theta(x,t) \\ 0 \\ r(x,\theta,t) \\ 1
    \end{pmatrix}\,,\; x'(0) = x'_0 := \begin{pmatrix}
        x_0 \\ \theta \\ 0 \\ t_0
    \end{pmatrix} \,.
\end{equation}
Next, define the augmented cost $C_{\text{aug}}:\mathcal{M}'\rightarrow \mathbb{R}$:
\begin{equation}
    C_{\text{aug}}(x') = F(x) + L\,.
\end{equation}
Then Equation \eqref{eq:single_trajectory_cost_mfd} can be rewritten as
\begin{equation}
    C^T_{f_\theta}(x'_0) = (C_{\text{aug}} \circ \Psi_{f_{\text{aug}}}^T)(x'_0)\,,
\end{equation}
which is of the required form to apply Theorem \ref{thm:adjoint_sensitivity_manifold}. By Theorem \ref{thm:adjoint_sensitivity_manifold}, the gradient $\text{d}\big(C_{\text{aug}} \circ \Psi_{f_{\text{aug}}}^T\big)$ is given by
\begin{equation} \label{eq:grad_aug_cost_mfd}
    \text{d} (C_{\text{aug}} \circ \Psi_{f_{\text{aug}}}^T)(x'_0) = \lambda(0) \,,
\end{equation}
where $\lambda(s)$ satisfies
\begin{align}\label{eq:lambda_aug_mfd}
    \dot{\lambda}_i = - \lambda_j \frac{\partial}{\partial x'_i} f_{\text{aug}}^j \,,\; \lambda(T) = \text{d}C_{\text{aug}}\big(x'(T)\big)\,.
\end{align}
Denote by $\text{d}_y,  \text{d}_\theta, \text{d}_L, \text{d}_t$ the components of the differential $\text{d}$ with respect to the component-manifolds $\mathcal{M}, \mathcal{P},\mathbb{R},\mathbb{R}$ of $\mathcal{M}'$, respectively, and similarly denote by $\lambda_x, \lambda_\theta, \lambda_L, \lambda_t$ the components of $\lambda$. The parameter gradient $\text{d}_\theta C^T_{f_\theta}\big((x_0,t_0),\theta\big)$ is then a component of the augmented cost gradient \eqref{eq:grad_aug_cost_mfd}:
\begin{equation}\label{eq:param_grad_aug_cost}
    \text{d}_\theta C^T_{f_\theta}\big((x_0,t_0),\theta\big) = \text{d}_\theta \big( C_\text{aug}\circ\Psi_{f_{\text{aug}}}^T\big)(x'_0) = \lambda_{\theta}(0)\,, 
\end{equation}
The dynamics of the components of the adjoint state are retrieved by expanding Equation\eqref{eq:lambda_aug_mfd}:
\begin{align}
    \dot{\lambda}_x &= - \frac{\partial}{\partial x} \big(\lambda_x (f_\theta) + \lambda_L(r)\big)
    \,,\; \lambda_x(T) = \text{d}_x C_{\text{aug}} = \text{d}_x F \,,
    \label{eq:lambda_x_mfd}
    \\
    \dot{\lambda}_\theta &= - \frac{\partial}{\partial \theta} \big(\lambda_x (f_\theta) + \lambda_L(r)\big)
    \,,\; \lambda_\theta(T) = \text{d}_\theta  C_{\text{aug}} = \text{d}_\theta F \,,
    \label{eq:lambda_theta_mfd}
    \\
    \dot{\lambda}_L &= 0
    \,,\; \lambda_L(T) = \text{d}_L C_{\text{aug}} = 1\,,
    \label{eq:lambda_L_mfd}
    \\
    \dot{\lambda}_t &= - \frac{\partial}{\partial t} \big(\lambda_x (f_\theta) + \lambda_L(r)\big)
    \,,\; \lambda_t(T) = \text{d}_t F = 0\,.
    \label{eq:lambda_t_mfd}
\end{align}
Note that $\lambda_L = 1$ is constant, such that equation \eqref{eq:lambda_x_mfd} coincides with \eqref{eq:adjoint_method_manifold}. Equation \eqref{eq:grad_cost_mfd} is recovered by integrating \eqref{eq:lambda_theta_mfd} from $s=0$ to $s=T$ and combining this with equation \eqref{eq:param_grad_aug_cost}. $\lambda_t$ does not appear in any of the other equations, such that Equation \eqref{eq:lambda_t_mfd} may be ignored.
\qed
\end{proof}

\begin{remark}
    Note that the parameter gradient $\text{d}_\theta C^T_{f_\theta}\big((x_0,t_0),\theta\big) \in T^*_\theta \mathcal{P}$ is a co-vector on $\mathcal{P}$. In order to recover a vector, a choice of metric tensor $M \in \Gamma(T^0_2\mathcal{P})$ is required. Then gradient descent corresponds to following integral curves of $M^{-1} \text{d}_\theta C^T_{f_\theta} \in \Gamma(T\mathcal{P})$. To recover the formalism in the main text we chose $\mathcal{P} = \mathbb{R}^k$ and $M = I_k$ 
    the $k$ by $k$ identity matrix in canonical coordinates on $\mathbb{R}^k$.
\end{remark}

Theorem \ref{thm:adjoint_method_mfd} also has a Hamiltonian form:

\begin{lemma}[Hamiltonian form of the generalized adjoint method]\label{lemma:Hamiltonian_adjoint_method_mfd}
   Define the time-dependent control Hamiltonian $H_c:T^*\mathcal{M}\times\times\mathcal{P}\times\mathbb{R}\rightarrow\mathbb{R}$ as
    \begin{equation}\label{eq:control_Hamiltonian_mfd}
        H_c(x,\lambda,\theta,t) = \lambda\big(f_\theta(x,t)\big) + r(x,\theta,t)\,.
    \end{equation}
   Then the integral equation \eqref{eq:grad_cost_mfd} reads
    \begin{equation}\label{eq:hamiltonian_grad_cost_mfd}
        \text{d}_\theta C^T_{f_\theta}(x,\theta) = \text{d}_\theta F + \int_0^T\text{d}_\theta H_c \text{d}t\,,
    \end{equation}
   and Hamilton's equations \eqref{eq:Ham_q} and \eqref{eq:Ham_p} agree with the equations for the adjoint sensitivity \eqref{eq:adjoint_method_manifold_x} and \eqref{eq:adjoint_method_manifold}, respectively:
    \begin{align}
        \dot{x}^j &=\frac{\partial H_c}{\partial \lambda_j} = f^j\,,\; 
        \label{eq:hamiltonian_general_adjoint_sensitivity_manifold_q} \\
        \dot{\lambda}_{i} &= - \frac{\partial H_c}{\partial x^i} \label{eq:hamiltonian_general_adjoint_sensitivity_manifold_p} 
        = - \lambda_{j} \frac{\partial}{\partial x^i} f^j  - \frac{\partial r}{\partial x^i}\,. 
    \end{align}
\end{lemma}

\subsection{Neural ODEs on Lie groups}\label{App:A:nODEs_on_Lie_groups}

The generalized adjoint method on a Lie group $G$ is obtained from the generalized adjoint method on manifolds. In the setup, consider the Lie group $G$ as a manifold, and consider that a dynamic system \eqref{app:eq:parameterized_f_mfd} and a cost \eqref{eq:single_trajectory_cost_mfd} are given.  




\begin{theorem}[The generalized adjoint method on Lie groups]\label{thm:hamiltonian_adjoint_method_lie_groups}
    Given a vector field $f_\theta \in \Gamma(TG)$ and a cost \eqref{eq:single_trajectory_cost_mfd}. Denote by $\tilde{f}_\theta(g,t) := (L_{g^{-1}})_* f_\theta(g,t)$ and define the control-Hamiltonian $\mathcal{H}_c:G\times\mathfrak{g}^*\times\mathcal{P}\times\mathbb{R}\rightarrow \mathbb{R}$ as 
    \begin{equation} \label{eq:control_Hamiltonian_Lie_group}
        \mathcal{H}_c(g,\bar{\lambda}_{g},\theta,t) = \bar{\lambda}_{g} (\tilde{f}_\theta(g,t)) + r(g,\theta,t) \,.
    \end{equation}
    Then the parameter gradient $\text{d}_\theta C^T_{f_\theta}(g_0,\theta) \in T^*_\theta \mathcal{P}$ is computed by 
    \begin{equation} \label{eq:grad_cost_lie_group}
       \text{d}_\theta C^T_{f_\theta}(g_0,\theta) 
       = 
        \text{d}_\theta F + \int_0^T \text{d}_\theta \mathcal{H}_c \text{d}t\,,
    \end{equation}
    where the state $g(t) \in G$ and adjoint state $\bar{\lambda}_g(t) \in \mathfrak{g}^*$ satisfy
    \begin{align}
        \dot{g} & = f_\theta\,, \; g(0) = g_0 \,, \label{eq:Lie_Adjoint_Sensitivity}  \\
        \dot{\bar{\lambda}}_{g} &= - \bar{\text{d}}_g(\bar{\lambda}_g(\tilde{f}_\theta)+r) 
        + \text{ad}_{\tilde{f}_\theta}^* \bar{\lambda}_g \,, \; \bar{\lambda}_{g}(T) = \big(\bar{\text{d}}_g F\big)\big(g(T)\big) \label{eq:Lie_Adjoint_Sensitivity_Co}
    \end{align}
\end{theorem}

\begin{proof}
    According to Equation \eqref{eq:HamLie_to_Ham} the Lie group control-Hamiltonian $\mathcal{H}_c:G\times\mathfrak{g}^*\times\mathcal{P}\times\mathbb{R}\rightarrow \mathbb{R}$ in \eqref{eq:control_Hamiltonian_Lie_group} directly corresponds to a manifold control-Hamiltonian $H_c:T^*G\times\mathcal{P}\times\mathbb{R}\rightarrow \mathbb{R}$ by substituting $\bar{\lambda}_g(t) = L^*_g\lambda(t)$:
    \begin{equation}\label{eq:Ham_mfd_to_Ham_Lie}
        H_c(g,\lambda,\theta,t) = \mathcal{H}_c(g,L_g^*\lambda,\theta,t) = \lambda\big(f_\theta(g,t)\big) + r(g,\theta,t)\,.
    \end{equation}
    By Lemma \ref{lemma:Hamiltonian_adjoint_method_mfd}, $H_c$ can be used to construct the adjoint sensitivity on $G$ as a manifold. By substitution of \eqref{eq:Ham_mfd_to_Ham_Lie} into \eqref{eq:hamiltonian_grad_cost_mfd}, equation \eqref{eq:grad_cost_lie_group} is recovered. 
    Further, Hamilton's equations \eqref{eq:hamiltonian_general_adjoint_sensitivity_manifold_q},\eqref{eq:hamiltonian_general_adjoint_sensitivity_manifold_p} are rewritten in their form on a Lie group by means of \eqref{eq:Lie_Ham},\eqref{eq:Lie_Ham_Co}: 
    \begin{align}
         \dot{g} &= {L_{g}}_*(\frac{\partial \mathcal{H}_c}{\partial \bar{\lambda}_{g}})  
         = f_\theta\,, \\
         \dot{\bar{\lambda}}_{g} &= - {L_{g}}^{*}(\frac{\partial \mathcal{H}_c}{\partial g}) + \text{ad}_{\frac{\partial \mathcal{H}_c}{\partial \bar{\lambda}_{g}}}^* \bar{\lambda}_{g}  \\
        & =  - \bar{\text{d}}_g(\bar{\lambda}_g(\tilde{f}_\theta)+r) 
        + \text{ad}_{\tilde{f}_\theta}^* \bar{\lambda}_g \,,
    \end{align}
    This recovers equations \eqref{eq:Lie_Adjoint_Sensitivity} and \eqref{eq:Lie_Adjoint_Sensitivity_Co}. To find the final condition $\bar{\lambda}_g(T)$, again use that $\bar{\lambda}_g(t) = L^*_g\lambda(t)$:
    \begin{equation}
        \bar{\lambda}_g(T) = L^*_g\lambda(T) = L^*_g \text{d}F(g) = \bar{d}_g F
    \end{equation}
    where the final step uses the definition of $\bar{d}_g$ in equation \eqref{eq:bar_grad_G}.
    \qed
\end{proof}

In order to arrive at the Hamiltonian form on matrix Lie groups, one may rewrite Equations \eqref{eq:hamiltonian_general_adjoint_sensitivity_manifold_q} and \eqref{eq:hamiltonian_general_adjoint_sensitivity_manifold_p} in their Hamiltonian form on matrix Lie groups \eqref{eq:Lie_Ham_Matrix}, \eqref{eq:Lie_Ham_Co_Matrix}.

\begin{remark}
    For a given finite-dimensional abstract Lie group the equations 
    \eqref{eq:Lie_Adjoint_Sensitivity} and \eqref{eq:Lie_Adjoint_Sensitivity_Co} two technical tools of Section \ref{sec:technical_tools} require a light adjustment. The Lie group integrator as in Section \eqref{ssec:technical_tools:lie_group_integrators} and the gradient $\text{d}_g$ presented in \eqref{ssec:technical_tools:gradients} should be phrased using the exponential map on the abstract Lie group \citep[Chapter 4.2.3]{Isham1999}. The chart dynamics \eqref{eq:dyn_Lie} should be phrased more generally with $\tilde{f}_{\theta}(g,t) := {L_{g^{-1}}}_* f_\theta(g,t)$. 
    Remaining operators such as $\Lambda$ in equation \eqref{eq:abstract_tilde} and $K(q_j)$ in equation \eqref{eq:dexp_rep} do not require adjustment, and can still be applied to render the dynamics of the state $g$ and adjoint state $\lambda_g$ on $\mathbb{R}^n$.
\end{remark}

\begin{remark}
While Theorem \ref{thm:adjoint_method_matrix_lie_group} covers the case for left-translated vector fields $\tilde{f}^L = (L_{g^{-1}})_* f$, where $\lambda$ is dual to the left-translation of $\dot{g}$, the case for right-translated vector fields $\tilde{f}^R = (R_{g^{-1}})_* f$ is nearly equivalent. Use of the right representation requires an adapted definition of the gradient operator \eqref{eq:grad_G} as
\begin{align}\label{eq:grad_G_R}
    \text{d}_g^R V & = \frac{\partial}{\partial q_j} V(\exp(\Lambda_e {q}_j) g)_{|q_j = 0} \nonumber \\ &= \frac{\partial}{\partial q_j} V((I+\Lambda{q}_j)g)_{|q_j = 0} \,. 
\end{align}
and the $\text{ad}_{\tilde{f}}$ term in the adjoint equation undergoes a sign-flip, leading to 
\begin{equation}
    \dot{\lambda}_g^R = \text{d}_g^R(\lambda_g^R(\tilde{f}_\theta^R)+r) 
        - \text{ad}_{\tilde{f}^R_\theta}^* \lambda_g^R \,. \label{eq:Lie_Adjoint_Sensitivity_Co_R}
\end{equation}
\end{remark}

\section{Theory}\label{sec:app}

\subsection{Chart invariance of derivative of exponential map}\label{appA:chart_invariance_K}
This appendix shows that \eqref{eq:G_vec_to_algebra_rep} is invariant under the choice of chart $(U_h,X_h)$, given that the chart is chosen from an exponential atlas \eqref{eq:one_chart_G} - \eqref{eq:one_chart_G_X_inv}.

The derivative of the exponential map $g(t) = \exp\big(\Lambda\big(q(t)\big)\big)$ is \citep{rossmann2006lie} 
\begin{equation}
    \dot{g} = \frac{\text{d}}{\text{d}t}\exp\big(\Lambda(q(t))\big)
\end{equation}
This can be represented as $\tilde{A} = (g^{-1}\dot{g}) \in \mathfrak{g}$ by
\begin{equation}
    \Tilde{A} = \bigg(\exp\big(\Lambda\big(q(t)\big)\big)^{-1} \frac{\text{d}}{\text{d}t}\exp\big(\Lambda(q(t))\big)\bigg) = 
     \Lambda\big(K(q)\dot{q}\big) \,,
\end{equation}
where $K(q)$ is as in equation \eqref{eq:dexp_rep}.  

In contrast, define $g(t) = X^{-1}_h\big(q_h(t)\big) = h \exp\big(\Lambda\big(q_h(t)\big)\big)$ in an exponential chart $(U_h,X_h)$. Then one finds 
\begin{equation}
    \dot{g} = \big(\frac{\text{d}}{\text{d}t} X^{-1}_h(q_h)\big) = h \frac{\text{d}}{\text{d}t} \exp\big(\Lambda\big(q_h(t)\big)\big) \,.
\end{equation}
Represent this as  $\Tilde{A} = \big(g^{-1} \dot{g}\big)$, and the expression is independent of $h$: 
\begin{equation}
    \Tilde{A} =  
    \bigg(\exp\big(\Lambda(q_h)\big)^{-1} \frac{\text{d}}{\text{d}t} \exp\big(\Lambda(q_h)\big) \dot{q}_h\bigg) = \Lambda\big(K(q_h) \dot{q}_h\big)\,.
\end{equation}


\subsection{Algorithm for chart-transitions}\label{ssec:app_ChartTransitions}

We describe an algorithm for computing the integral curves $g:\mathbb{R}\rightarrow G$ of a vector field $f\in\Gamma(TG)$, by computations in local charts $(U_i, X_i)$ from a minimal atlas $\mathcal{A}^G_{\text{min}}$ of $G$. 

In the algorithm, the trajectory $g(t)$ and the vector field $f$ are represented in terms of a chart-representative $q_i(t) := X_i \circ g(t) \in \mathbb{R}^n$ and $f_i := {X_i}_*f \in \Gamma(T\mathbb{R}^n)$, respectively. Integration goes from $t = 0$ to $t = T$. An integration step from $q_i(t)$ to $q_i(t+\Delta t)$ consists of computing the flow-map $\Psi^{\Delta t}_{f_i}:\mathbb{R}^n \rightarrow \mathbb{R}^n$, which corresponds to applying an (arbitrary) ODE-solver with initial condition $q_i(t)$. 

Denote by functions $\sigma_i:G \rightarrow \mathbb{R}$ a partition of unity w.r.t. the chart regions $U_i$ of $\mathcal{A}^G_{\text{min}}$, i.e., the $\sigma_i$ satisfy 
\begin{align}
    &0 \leq \sigma_i(g) \leq 1 \,, \label{eq:PoU-1}\\
    &\sum_i \sigma_i(g) = 1 \,, \label{eq:PoU-2} \\ 
    &\sigma_i(g) \geq 0 \Leftrightarrow g \in U_i \label{eq:PoU-3} \,.
\end{align}
To determine the index $i$ for a given integration step, we choose $i = \arg\max_i \sigma\big(g(t)\big)$. When $\sigma_i\big(g(t)\big)$ falls below a threshold-value $\sigma_{\text{min}}$ during integration, the chart is switched to the one with the largest $\sigma_i$. With $N$ the number of charts in $\mathcal{A}^G_{\text{min}}$, we choose $\sigma_\text{min} = 1/(1+N)$. 


The full trajectory is stored in an array $Q$ in terms of chart components $(q_i,i) \in \mathbb{R}^{n+1}$.

The algorithm is summarized in Algorithm \ref{app:alg_chart_switch_Lie_group}: 

\begin{algorithm}
    \caption{Chart-Switching on G}\label{app:alg_chart_switch_Lie_group}
    \begin{algorithmic}[1]
        \State $i \gets \arg \max_i \sigma_i(g_0)$
        \State $q_{i} \gets X_i (g_0)$
        \State $Q \gets (q_i, i)$
        \State $t \gets 0$
        \While{$t < T$}
            \State $t  \gets t+\Delta t$
            \State $q_{i} \gets \Psi^{\Delta t}_{f_i}(q_i)$
            \If{$\sigma_i(X_i(q_i)) < \sigma_\text{min}$}
                \State $i_+ \gets \arg \max_i \sigma_i(X_i^{-1}(q_i))$
                \State $q_{i} \gets X_{i_+}\circ X_i^{-1} (q_i)$
                \State $i \gets i_+$
            \EndIf
            \State $Q \gets \text{concatenate}(Q,(q_i,i))$
        \EndWhile
    \end{algorithmic}
\end{algorithm}

\begin{remark}
    The choice of $\sigma_\text{min}$ requires an atlas with a locally finite number of charts. This is possible as long as $G$ is paracompact. Then we choose $\sigma_\text{min} = 1/(1+N(g))$, where $N(g) = \sum_k (\sigma_k(X_i^{-1}(q_i)) > 0)$ is the number of non-zero $\sigma_i$ at $g$. 
\end{remark}

For $G = SE(3)$ and $\mathcal{A}^{SE(3)}_{\text{min}}$ chosen as in \eqref{eq:MinAtlas_SE3}, a partition of unity is given by \eqref{eq:PoU_SE3}. For this choice of atlas and partition of unity the cut-off $\sigma_\text{min} = 1/5$ is assigned. 


\subsection{Completeness of Minimal Atlas} \label{ssec:app_MinAt_Complete}
This section shows that $\mathcal{A}^{SE(3)}_\text{min}$ and  $\mathcal{A}^{SO(3)}_\text{min}$ are complete atlases.  
Recall that the charts of $\mathcal{A}^{SO(3)}_\text{min}$ are $(U_i, X_i)$ with $i = 0,1,2,3$ and chart-regions $U_i$ defined by \eqref{eq:chart_region_SO3}.
We will show that $\bigcup_i U_{i} = SO(3)$. 

To this end, we aim to find a partition of unity w.r.t. the chart regions of $\mathcal{A}^{SO(3)}_\text{min}$, i.e., partition functions $\sigma_i:SO(3)\rightarrow \mathbb{R}$ that satisfy \eqref{eq:PoU-1}, \eqref{eq:PoU-2} and \eqref{eq:PoU-3}.

Given such $\sigma_i$, properties \eqref{eq:PoU-2} and \eqref{eq:PoU-3} directly imply that $\bigcup U_i = SO(3)$: for any $R \in SO(3)$ there must be an $i$ such that $\sigma_i(R) > 0$, hence there be a $U_i$ such that $R \in U_i$. Thus $SO(3) \subseteq \bigcup_i U_i$, and the converse $\bigcup_i U_i \subseteq SO(3)$ holds trivially, such that $SO(3) = \bigcup_i U_i$.

Consider the following candidate partition of unity w.r.t. $U_i$:
\begin{equation}\label{eq:PoU_SO3}
    \sigma_i(R) = \big(\text{Tr}(R_{i}^{\top}R)+1\big)/4\,.
\end{equation}
The remainder of the proof consists of showing that \eqref{eq:PoU_SO3} indeed constitute a partition of unity, i.e., that they satisfy \eqref{eq:PoU-1}, \eqref{eq:PoU-2} and \eqref{eq:PoU-3}.

To this end, express $R$ in terms of coordinates $\omega_i \in \mathbb{R}^3$ in the chart $(U_i,X_i)$, i.e. with $\theta_i = \sqrt{\omega_i^\top \omega_i}$ and $\hat{\omega_i} = \frac{\omega_i}{\theta_i}$, as 
\begin{align} \label{eq:R_in_chart_i}
    R(\omega_i) & = X_i^{-1}(\omega_i) = R_i e^{\Tilde{\omega}_i} \\ 
    & = I + \sin(\theta_i) \Tilde{\hat{\omega}}_i + \big(1-\cos(\theta_i)\big)\Tilde{\hat{\omega}}_i^2 \nonumber \,.   
\end{align}

Combine \eqref{eq:PoU_SO3} and \eqref{eq:R_in_chart_i} to find
\begin{equation} \label{eq:PoU_in_Chart}
    \sigma_i\big(R(\omega_i)\big) = \big(\text{Tr}(e^{\Tilde{\omega}_i})+1\big)/4 = 1/2 + 1/2\cos(\theta_i)\,.
\end{equation}
Hence, all $\sigma_i$ are bounded by $0$ and $1$, fulfilling property \eqref{eq:PoU-1}. It is also clear that $\sigma_i(R) = 0$ when $\theta_i = \pi$, which is the condition for $R \neq U_i$, such that all $\sigma_i$ fulfill property \eqref{eq:PoU-3}. 

To show that condition \eqref{eq:PoU-2} is fulfilled, express $R$ in the chart $(U_0,X_0)$, i.e.,
\begin{equation}
    R(\omega_0) = e^{\tilde{\omega}_0}\,,
\end{equation}
and compute 
\begin{align} 
    \text{Tr}\big(R_{0}^{\top}R(\omega_0)\big) &= \text{Tr}(R) = 1+2\cos(\theta_0) \label{eqTr0}\,, \\
    \text{Tr}\big(R_{1}^{\top}R(\omega_0)\big) &= -1 + 2(1-\cos(\theta_0))\hat{\omega}_{0,1}^2 \label{eqTr1}\,,\\
    \text{Tr}\big(R_{2}^{\top}R(\omega_0)\big) &= -1 + 2(1-\cos(\theta_0))\hat{\omega}_{0,2}^2 \label{eqTr2}\,,\\
    \text{Tr}\big(R_{3}^{\top}R(\omega_0)\big) &= -1 + 2(1-\cos(\theta_0))\hat{\omega}_{0,3}^2 \label{eqTr3} \,.
\end{align}
Then
\begin{align}
    \sum_i \sigma_i\big(R(\omega_0)\big) &=  \sum_i (\text{Tr}(R_{i}^{\top}R(\omega_0)+1)/4 = 1 \,, 
\end{align}
which confirms that the $\sigma_i$ fulfill property \eqref{eq:PoU-2}.
Hence, there is always a valid chart given by $\sigma_i(R)>0$ and $\mathcal{A}_{\min}$ is indeed a minimal Atlas.

The same argument applies to show the completeness of $\mathcal{A}^{SE(3)}_\text{min}$: indeed, the functions 
\begin{equation}\label{eq:PoU_SE3}
    \sigma_i(H) = \text{Tr}(H_{i}^{-1}H)/4\,,
\end{equation}
constitute a partition of unity w.r.t. the chart regions $\mathcal{U}_i$ defined by \eqref{eq:chart_region_SE3}.


\subsection{Derivative of the exponential map on $SE(3)$}\label{ssec:app:derivative_exp_SE3}
Denoting $q = (\omega, v)^T$ and $\|\omega_i\|_2 = \theta$, the derivative $K(q)$ of the exponential map is found via the Caley-Hamilton theorem \citep{Visser2006} as
\begin{equation}
    K(q) = \sum_{i=0}^5 a_i(\theta)\text{ad}_q^i\,.
\end{equation}

Let $\text{sinc}(\theta) := \sin(\theta)/\theta$, then the $a_i(\theta)$ are given by
\begin{align*}
    a_0 &= 1 \,, \\
    a_1 &= -\frac{1}{2} \,, \\
    a_2 &= \frac{1}{4\theta^2} (8 + 2\cos(\theta) - 10 \,\text{sinc}(\theta)) \,, \\
    a_3 &= \frac{1}{4\theta^3} (-4\theta + \frac{12 - 12\cos(\theta)}{\theta} - 2\sin(\theta)) \,, \\
    a_4 &= \frac{1}{4\theta^4} (4 + 2\cos(\theta) - 6 \,\text{sinc}(\theta)) \,, \\
    a_5 &= \frac{1}{4\theta^5} (-2\theta + \frac{8 - 8\cos(\theta)}{\theta} - 2\sin(\theta))\,,
\end{align*}

\subsection{Construction of additional Lie groups} \label{app:additional_Lie_groups}

The Lie groups $(\mathbb{R}^k,+)$ and $SE(3)\times \mathbb{R}^6$ are constructed as matrix Lie groups.

\subsubsection{The Lie group $(\mathbb{R}^k,+)$}\label{app:Lie_group_(R^k,+)}

The Lie group $(\mathbb{R}^k,+)$ is defined as a matrix Lie group $\text{Vec}(k,\mathbb{R}) \subset GL(k+1,\mathbb{R})$ by
\begin{equation}
    \text{Vec}(k,\mathbb{R}) = \{\begin{bmatrix} I & p \\ 0& 1 \end{bmatrix} \in \mathbb{R}^{(k+1)\times (k+1)}| \, p \in \mathbb{R}^k\}\,,
\end{equation}
with Lie algebra $\text{vec}(k,\mathbb{R}) \subset gl(k+1,\mathbb{R})$
\begin{equation}
    \text{vec}(k,\mathbb{R}) = \{\begin{bmatrix} 0 & v \\ 0& 0 \end{bmatrix} \in \mathbb{R}^{(k+1)\times (k+1)}| \, v \in \mathbb{R}^k\}\,.
\end{equation}
Define 
\begin{equation}
    \Lambda:\mathbb{R}^k \rightarrow \text{vec}(k,\mathbb{R})\,;\; v \mapsto \begin{bmatrix} 0 & v \\ 0& 0 \end{bmatrix}\,,
\end{equation}
then $\text{ad}_v \in \mathbb{R}^{k \times k} $ is the zero matrix.
The exponential and logarithmic maps are
\begin{align}
    &\exp(\Lambda(v)) = e^{\Lambda(v)} = I + \Lambda(v) \,, \label{eq:exp_Vecn}\\
    &\log(\begin{bmatrix} I & v \\ 0& 1 \end{bmatrix}) = v \,,
\end{align}
This logarithmic map has a global range, such that $\mathcal{A}
= \{(\text{Vec}(k,\mathbb{R}), \log) \}$ is a minimal exponential atlas for $\text{Vec}(k,\mathbb{R})$.
By definition \eqref{eq:grad_G} the gradient of $V:\text{Vec}(k,\mathbb{R})\rightarrow \mathbb{R}$ at $g = \begin{bmatrix} I & p \\ 0& 1 \end{bmatrix}$ reads
\begin{equation}
    d_g V = \frac{\partial}{\partial p} V(\begin{bmatrix} I & p \\ 0& 1 \end{bmatrix}) \,.
\end{equation}


\subsubsection{The Lie group $SE(3)\times \mathbb{R}^6$}\label{app:Lie_group_SE3_x_R^6}

We show a detailed construction of the Lie group $G = SE(3) \times se^*(3)$ as a matrix Lie group $G\subset GL(11,\mathbb{R})$.

In order to associate $se^*(3)$ with $Vec(6,\mathbb{R})$, note that $\Lambda^*:se^*(3) \rightarrow \mathbb{R}^6$ is a homomorphism from $(se^*(3),+)$ to $(\mathbb{R}^6,+)$, i.e. $\Lambda(P_1+P_2) = \Lambda(P_1)+\Lambda(P_2)$. By the construction of $Vec(n,\mathbb{R})$ from $(\mathbb{R}^k,+)$ in Section \ref{app:Lie_group_(R^k,+)}, we can define the homomorphism 
\begin{equation}
    \hookrightarrow:se^*(3)\rightarrow GL(7,\mathbb{R})\,;\; P\mapsto \begin{bmatrix}
        I & \Lambda(P) \\ 0 & 1
    \end{bmatrix}\,.
\end{equation}

And one can compose $SE(3)\subset GL(4,\mathbb{R})$ and $Vec(6,\mathbb{R}) \subset GL(7,\mathbb{R})$ 
\begin{equation}
    G := \{ \begin{bmatrix} H & 0 & 0 \\ 0 & I & P \\ 0& 0 & 1  \end{bmatrix} \, | \, H \in SE(3), P \in \mathbb{R}^6 \}\,.
\end{equation}
Its Lie algebra $\mathfrak{g} \subset gl(11,\mathbb{R})$ is 
\begin{equation}
    \mathfrak{g} = \{ \begin{bmatrix} \tilde{T} & 0 & 0 \\ 0 & 0 & P \\ 0& 0 & 0  \end{bmatrix} \, | \, \tilde{T} \in se(3), P \in \mathbb{R}^6 \}\,.
\end{equation}
For $A = \begin{pmatrix}
        T \\ P
    \end{pmatrix}) \in \mathbb{R}^{12}$ choose $\Lambda:\mathbb{R}^{12}\rightarrow\mathfrak{g}$ as
\begin{equation}
    \Lambda(
    \begin{pmatrix}
        T \\ P
    \end{pmatrix}) = 
    \begin{bmatrix} 
        \Lambda(T) & 0 & 0 \\ 
        0 & 0 & P \\ 
        0& 0 & 0  
    \end{bmatrix}\,,
\end{equation}
and the adjoint map is 
\begin{equation}
    \text{ad}_A = 
    \begin{bmatrix}
    \text{ad}_T & 0 \\ 0 & 0     
    \end{bmatrix}\,.
\end{equation}
The exponential map on $G$ is 
\begin{align}
    \exp:&\,\mathfrak{g} \rightarrow G\,;\; \begin{bmatrix} \tilde{T} & 0 & 0 \\ 0 & 0 & P \\ 0& 0 & 0  \end{bmatrix} \mapsto \begin{bmatrix} e^{\tilde{T}} & 0 & 0 \\ 0 & I & P \\ 0& 0 & 1  \end{bmatrix} \,.
\end{align}
Note that the exponential map is composed of the exponential maps on $SE(3)$ \eqref{eq:expSE3} and $Vec(6,\mathbb{R})$ \eqref{eq:exp_Vecn},
\begin{equation}
    \exp(\Lambda(A)) = \begin{bmatrix}
    \exp(\Lambda(T)) & 0 \\ 0 & \exp(\Lambda(P))
\end{bmatrix}\,.
\end{equation}
Thus, the differential operator $\text{d}_\Gamma:C^1(G,\mathbb{R})\rightarrow \mathbb{R}$ can be split as
\begin{align}
    \text{d}_\Gamma V(\Gamma) &= \frac{\partial}{\partial A} V(g\exp(\Lambda(A))) \\
    &= 
    \begin{pmatrix}
        \frac{\partial}{\partial q} \\ \frac{\partial}{\partial p}
    \end{pmatrix} V( g 
    \begin{bmatrix}
    \exp(\Lambda(T)) & 0 \\ 0 & \exp(\Lambda(P))
\end{bmatrix}) \nonumber \\
& = \begin{pmatrix}
        \text{d}_H \\ \text{d}_P
    \end{pmatrix} V(H,P) \,.
\end{align}
where we define $V(H,P) := V(\begin{bmatrix} H & 0 & 0 \\ 0 & I & P \\ 0& 0 & 1  \end{bmatrix})$.

\section{Training}\label{sec:appB} 
\subsection{Hyperparameters of Training}\label{ssec:appB:Hyperparameters}
Table \ref{tab:appB:Hyperparameters_quadratic_training}, Table \ref{tab:appB:Hyperparameters_nonlinear_training} and Table \ref{tab:appB:Hyperparameters_nonlinear_training_gravity} summarize the hyper parameters used for the training in Section \ref{ssec:train:quadratic_training}, Section \ref{ssec:train:nonlinear_training} and Section \ref{ssec:train:gravity}, respectively.

\begin{table}[]
    \small\sf\centering
    \caption{Hyperparameters corresponding to optimizing the quadratic controller in Section \ref{ssec:train:quadratic_training}.}
    \label{tab:appB:Hyperparameters_quadratic_training}
    \begin{tabular}{|c|c|}
    \toprule
        Variable & Value \\
    \midrule
         $H_F$ 
         &  $\begin{bmatrix}
                I & 0 \\ 0 & 1
            \end{bmatrix}$ \\
         $T$
         & 3 \\
         $w_1$
         & 4 \\
         $w_2$
         & 20 \\
         $w_3$
         & 5 \\
         $w_4$
         & 1 \\
         $w_5$
         & 1 \\
         $w_6$
         & 1 \\
         $w_7$
         & 1\\
         $w_8$
         & 1 \\
         $w_9$
         & 1 \\
         Epochs & 1200 \\
         $\eta$ over first 1000 epochs 
         & $1e-3$ \\
         $\eta$ over final 200 epochs
         & $1e-2$ \\
         $\gamma$
         &  0.999\\
         Batch size 
         &  2048 \\
         ODE Solver 
         &  Dormand-Prince 5\\
         \texttt{rtol}
         &  $1e-5$\\
         \texttt{atol}
         &  $1e-4$\\
         \texttt{rtol\_adjoint}
         & $1e-5$ \\
         \texttt{atol\_adjoint}
         & $1e-4$\\
    \bottomrule
    \end{tabular}
\end{table}


\begin{table}[]
    \small\sf\centering
    \caption{Hyperparameters corresponding to optimizing the nonlinear controller in Section \ref{ssec:train:nonlinear_training}.}
    \label{tab:appB:Hyperparameters_nonlinear_training}
    \begin{tabular}{|c|c|}
    \toprule
    Variable & Value \\
    \midrule
         $H_F$ 
         &  $\begin{bmatrix}
                I & 0 \\ 0 & 1
            \end{bmatrix}$ \\
         $T$
         & 3 \\
         $w_1$
         & 4 \\
         $w_2$
         & 10 \\
         $w_3$
         & 5 \\
         $w_4$
         & 1 \\
         $w_5$
         & 1 \\
         $w_6$
         & 1 \\
         $w_7$
         & 1\\
         $w_8$
         & 1 \\
         $w_9$
         & 1 \\
         Epochs & 1200 \\
         $\eta$ over first 1000 epochs 
         & $1e-3$ \\
         $\eta$ over final 200 epochs
         & $1e-3$ \\
         $\gamma$
         &  0.999\\
         Batch size 
         &  2048 \\
         ODE Solver 
         &  Dormand-Prince 5\\
         \texttt{rtol}
         &  $1e-5$\\
         \texttt{atol}
         &  $1e-4$\\
         \texttt{rtol\_adjoint}
         & $1e-5$ \\
         \texttt{atol\_adjoint}
         & $1e-4$\\
    \bottomrule
    \end{tabular}
\end{table}

\begin{table}[]
    \small\sf\centering
    \caption{Hyperparameters corresponding to optimizing a nonlinear controller including gravity in the cost, as in Section \ref{ssec:train:gravity}.}
    \label{tab:appB:Hyperparameters_nonlinear_training_gravity}
    \begin{tabular}{|c|c|}
    \toprule
    Variable & Value \\
    \midrule
         $H_F$ 
         &  $\begin{bmatrix}
                I & p_F \\ 0 & 1
            \end{bmatrix}$ 
            with $p_F = (0,0,-1)^T$ \\
         $T$
         & 3 \\
         $w_1$
         & 4 \\
         $w_2$
         & 4 \\
         $w_3$
         & 5 \\
         $w_4$
         & $5e-4$ \\
         $w_5$
         & 1 \\
         $w_6$
         & 1 \\
         $w_7$
         & 1\\
         $w_8$
         & $1e-4$ \\
         $w_9$
         & 1 \\
         Epochs & 1000 \\
         $\eta$ 
         & $1e-3$ \\
         $\gamma$
         &  0.999\\
         Batch size 
         &  2048 \\
         ODE Solver 
         &  Dormand-Prince 5\\
         \texttt{rtol}
         &  $1e-5$\\
         \texttt{atol}
         &  $1e-4$\\
         \texttt{rtol\_adjoint}
         & $1e-5$ \\
         \texttt{atol\_adjoint}
         & $1e-4$\\
    \bottomrule
    \end{tabular}
\end{table}

\end{sm}

\end{document}